\documentclass[11pt]{article}
 \pdfoutput=1                           
\usepackage{a4wide}
\usepackage{amsmath,amsfonts,amssymb,amsthm,amscd}
\usepackage{epsfig}
\usepackage{hyperref}
\usepackage{breakurl}

\usepackage{color}
\usepackage{enumerate} 

\if10     
\usepackage[mathlines]{lineno}
\newcommand*\patchAmsMathEnvironmentForLineno[1]{%
  \expandafter\let\csname old#1\expandafter\endcsname\csname #1\endcsname
  \expandafter\let\csname oldend#1\expandafter\endcsname\csname end#1\endcsname
  \renewenvironment{#1}%
     {\linenomath\csname old#1\endcsname}%
     {\csname oldend#1\endcsname\endlinenomath}}%
\newcommand*\patchBothAmsMathEnvironmentsForLineno[1]{%
  \patchAmsMathEnvironmentForLineno{#1}%
  \patchAmsMathEnvironmentForLineno{#1*}}%
\AtBeginDocument{%
\patchBothAmsMathEnvironmentsForLineno{equation}%
\patchBothAmsMathEnvironmentsForLineno{align}%
\patchBothAmsMathEnvironmentsForLineno{flalign}%
\patchBothAmsMathEnvironmentsForLineno{alignat}%
\patchBothAmsMathEnvironmentsForLineno{gather}%
\patchBothAmsMathEnvironmentsForLineno{multline}%
}
\linenumbers
\fi

\definecolor{modra3}{rgb}{.1,.0,.4}
\hypersetup{colorlinks=true, linkcolor=modra3, urlcolor=modra3, citecolor=modra3, pdfpagemode=UseNone, pdfstartview=} 

\usepackage{ifpdf} 

\newtheorem{theorem}{Theorem}

\newtheorem{lemma}[theorem]{Lemma}
\newtheorem{corollary}[theorem]{Corollary}

\newtheorem{observation}[theorem]{Observation}
\newtheorem{problem}{Problem}
\newtheorem*{definition}{Definition}

\theoremstyle{remark}
\newtheorem*{remark}{Remark}

\long\def\onefigure#1#2#3#4{
\begin{figure}
\begin{center}
\scalebox{#2}
{
\ifpdf\includegraphics{#1}\fi
}
\end{center}
\caption{\label{#4} #3}
\end{figure}
} 

\newcommand{\myfig}[3]  
{\onefigure{{p-#1.pdf}} {#2} {#3} {fig:#1}}

\if10  
\fi

\newcommand{\Rbb}{\mathbb R}

\DeclareMathOperator{\conv}{conv}
\DeclareMathOperator{\relheight}{rh}

\def\inst#1{$^{#1}$}

\begin{document}

\title{On three measures of non-convexity\thanks{This research was supported by
the project CE-ITI (GA\v{C}R P202/12/G061) of the Czech Science Foundation.
The first, the third and the sixth author were partially supported by the Charles University Grant GAUK 52410. The first, the third and the fifth author were partially supported by the grant SVV-2013-267313 (Discrete Models and Algorithms).
The second author was supported by the project CZ.1.07/2.3.00/20.0003 of the Operational Programme Education
for Competitiveness of the Ministry of Education, Youth and Sports of the Czech Republic.
The third author was also partially supported by ERC Advanced Research Grant no 267165 (DISCONV).
The fourth author was partially supported by ESF EuroGiga
project ComPoSe (IP03), by OTKA Grant K76099 and by OTKA Grant 102029.
Part of the research was conducted during the Special Semester on Discrete
and Computational Geometry at \'Ecole Polytechnique F\'ederale de
Lausanne, organized and supported by the CIB (Centre  Interfacultaire
Bernoulli) and the SNSF (Swiss National Science Foundation).
} 
} 


\author{Josef Cibulka\inst{1} \and
Miroslav Korbel\'{a}\v{r}\inst{3} \and
Jan Kyn\v{c}l\inst{1}$^{, }$\inst{2} \and
Viola M\'{e}sz\'{a}ros\inst{4} \and
Rudolf Stola\v{r}\inst{1} \and
Pavel Valtr\inst{1}
} 

\date{}

\maketitle

\begin{center}
{\footnotesize
\inst{1}
Department of Applied Mathematics and Institute for Theoretical Computer Science, \\
Charles University, Faculty of Mathematics and Physics, \\
Malostransk\'e n\'am.~25, 118~00~ Prague, Czech Republic; \\
\texttt{\{cibulka,kyncl,ruda\}@kam.mff.cuni.cz}
\\\ \\
\inst{2}
Alfr\'ed R\'enyi Institute of Mathematics, Re\'altanoda u. 13-15, Budapest 1053, Hungary 
\\\ \\
\inst{3}
Department of Mathematics and Statistics, \\
 Faculty of Science, Masaryk University,\\
Kotl\' a\v rsk\'{a} 2, 611 37 Brno, Czech Republic; \\
\texttt{miroslav.korbelar@gmail.com}
\\\ \\
\inst{4}
University of Szeged, Bolyai Institute, \\ 
6720 Szeged, Aradi v{\'e}rtan{\'u}k tere 1, Hungary; \\
\texttt{viola@math.u-szeged.hu}
}
\end{center}  


\begin{abstract}
The invisibility graph $I(X)$ of a set $X \subseteq \Rbb^d$ is a (possibly infinite) graph whose vertices are
the points of $X$ and two vertices are connected by an edge if and only if the straight-line segment connecting
the two corresponding points is not fully contained in $X$.
We consider the following three parameters of a set $X$: the clique number $\omega(I(X))$, the chromatic number
$\chi(I(X))$ and the convexity number $\gamma(X)$, which is the minimum number of convex subsets of $X$ that cover $X$.

We settle a conjecture of Matou\v{s}ek and Valtr claiming that for every planar set $X$,
$\gamma(X)$ can be bounded in terms of $\chi(I(X))$.
As a part of the proof we show that a disc with $n$ one-point holes near its boundary has
$\chi(I(X)) \ge \log\log(n)$ but $\omega(I(X))=3$.

We also find sets $X$ in $\Rbb^5$ with $\chi(X)=2$, but $\gamma(X)$ arbitrarily large.

\end{abstract}

\section{Introduction}

The \emph{convexity number} $\gamma(X)$ of a set $X \subseteq \Rbb^d$
is the minimum possible number of convex subsets of $X$ that cover $X$. Intuitively, the larger the $\gamma(X)$, the more ``non-convex'' $X$ is. Two other parameters of non-convexity have been widely studied in the literature. We define them in terms of the invisibility graph of $X$. 

We say that two points $x,y \in X$ \emph{see each other (within $X$)}
if the straight-line segment $\overline{xy}$ connecting $x$ and $y$ is a subset of $X$.
The \emph{invisibility graph\/} $I(X)$ of $X$ is a graph whose vertices are
the points of $X$ and two vertices are connected by an edge if and only if they do not see each other. Clearly, a set $A\subseteq X$ is a clique in $I(X)$ if no two points from $A$ see each other within $X$. Analogously, $A$ is an independent set in $I(X)$ if every two points from $A$ see each other within $X$. An independent set of $I(X)$ is also called a \emph{seeing subset} of $X$.

The chromatic number of $X$, denoted by $\chi(X)$, is defined as the chromatic number of the invisibility graph $I(X)$. In other words, $\chi(X)$ is the minimum possible number of seeing subsets of $X$ that cover $X$. Similarly, the clique number of $X$, denoted by $\omega(X)$, is defined as the supremum of the cardinalities of cliques in $I(X)$. Note that if the clique number is finite, then a clique of maximum size always exists. 
Sets $X$ with $\omega(X)=n-1$ are also called \emph{$n$-convex}. 
The parameters $\omega(X)$ and $\chi(X)$ have been also denoted by $\alpha(X)$ and $\beta(X)$, respectively~\cite{KPS90,NiPe13,PerlesShelah90}.

Clearly, $\omega(X) \le \chi(X) \le \gamma(X)$ for any set $X$. 
We mostly consider sets with all the three parameters finite.
We are interested in the following questions. 
Is it possible to bound $\gamma(X)$ from above by a function of $\chi(X)$ or $\omega(X)$? If yes, what is the best such function? What if the set $X$ is ``nice'', for example, closed? How does the answer depend on the dimension $d$?  


For closed connected sets $X$ in the plane, Valentine~\cite{Valentine57} proved that $\omega(X)=2$ implies $\gamma(X)\le 3$. McKinney~\cite{McKinney66} proved that if $X$ is a closed set in $\mathbb{R}^d$, then $\chi(X)=2$ implies $\gamma(X)=2$.
Eggleston~\cite{eggleston74} proved that if $X$ is a compact set in the plane with $\omega(X)$ finite, then $\gamma(X)$ is finite as well.
Breen and Kay~\cite{BreenKay76} were first to show that for planar closed sets $X$, the number $\gamma(X)$ can be bounded by a function of $\omega(X)$. More precisely, they gave an exponential bound $\gamma(X)\le \omega(X)^3\cdot 2^{\omega(X)-2}$. 
Perles and Shelah~\cite{PerlesShelah90} further improved this to a polynomial upper bound $\gamma(X)\le \omega(X)^6$.
The current best known upper bound, 
$\gamma(X) \le 18\omega(X)^3$, is due to Matou\v{s}ek and Valtr~\cite{MatousekValtr99}. 

From the other direction, Breen and Kay~\cite{BreenKay76} presented a construction by Perles of a closed planar set $X$ with $\gamma(X)\ge \Omega(\omega(X)^{3/2})$.
Matou\v{s}ek and Valtr~\cite{MatousekValtr99} found examples of closed planar sets $X$ with $\gamma(X) \ge \Omega(\omega(X)^{2}$).

For arbitrary planar sets $X$, not necessarily closed, $\omega(X)=2$ implies $\gamma(X)\le 6$~\cite{breen74,NiPe13}.
However, there is no upper bound on $\gamma(X)$ for
sets with $\omega(X)=3$. To construct an example, take the unit disc and puncture $\lambda$ one-point holes near its boundary, in the
vertices of a regular convex $\lambda$-gon concentric with the disc. The resulting set $D_{\lambda}$ satisfies $\omega(D_{\lambda})=3$ and
$\gamma(D_{\lambda}) = \lceil \lambda/2 \rceil + 1$~\cite{MatousekValtr99}. Kojman, Perles and Shelah~\cite{KPS90} constructed a planar set $X$ with $\omega(X)=5$ and $\gamma(X) = 2^{\aleph_0}$.

Kojman, Perles and Shelah~\cite{KPS90} proved that if $X$ is a closed planar set with $\gamma(X)$ uncountable, then $X$ contains a perfect subset $P$ such that the convex hull of any three distinct points from $P$ is not contained in $X$. In particular, $P$ has cardinality $2^{\aleph_0}$ and hence also $\gamma(X) = 2^{\aleph_0}$.
Kojman~\cite{Ko01} generalized the notion of the clique in the invisibility graph as follows. For $m\ge 2$, a subset $P$ of $X$ is an \emph{$m$-clique} in $X$ if for every $m$-element subset $S$ of $P$ the convex hull of $S$ is not contained in $X$. Kojman~\cite{Ko01} showed that for every closed planar set $X$ with countable $\gamma(X)$, the topological complexity of all $3$-cliques in $X$, measured by the \emph{Cantor--Bendixson degree}, is bounded by a countable ordinal.
Geschke~\cite{G04} constructed a closed set $X$ in $\mathbb{R}^3$ with $\gamma(X)=2^{\aleph_0}$ such that for every $m\ge 2$, all $m$-cliques in $X$ are countable.

A \emph{one-point hole\/} in a set $X \subset \Rbb^d$ is a point that forms a path-connected component
of $\Rbb^d \setminus X$. Let $\lambda(X)$ be the number of one-point holes in $X$.

The example of the set $D_{\lambda}$ inspired Matou\v{s}ek and Valtr~\cite{MatousekValtr99} to study the properties of planar sets with a limited number of one-point
holes. In particular, they proved the following upper bound on the convexity number.

\begin{theorem}[{\cite[Theorem 1.1 (ii)]{MatousekValtr99}}]
\label{thm:MatVal}
Let $X \subseteq \Rbb^2$ be a set with $\omega(X)=\omega < \infty$ and $\lambda(X)=\lambda < \infty$. Then
\[
\gamma(X) \le O(\omega^4 + \lambda \omega^2).
\]
\end{theorem}


For any $\omega \ge 3$ and $\lambda \ge 0$ they also found sets $X$ with
$\omega(X) = \omega$, $\lambda(X)=\lambda$ and $\gamma(X) \ge \Omega(\omega^3 + \omega \lambda)$.

Matou\v{s}ek and Valtr~\cite{MatousekValtr99} conjectured that for an arbitrary planar set $X$, the value
of $\gamma(X)$ is bounded by a function of $\chi(X)$. Then $\chi(X)$ cannot be bounded by a function of
$\omega(X)$ as the examples $D_{\lambda}$ show.

Lawrence and Morris~\cite{LawrenceMorris09} proved that for every $k$ there exists $n_0(k) \le 2^{2^{O(k)}}$ 
such that whenever $S$ is a set of finitely many points in the plane and $|S| \ge n_0(k)$, 
then $\chi(\Rbb^2 \setminus S) \ge k$.\footnote{The graph
$\mathcal{G}_S$ in the paper of Lawrence and Morris is precisely the invisibility graph of $\Rbb^2 \setminus S$.
The proof of Theorem 6 is not correct in the case when $S$ is not in general position, 
because some of the points $q_{i,j}$ may coincide with some points from $S \setminus S'$, 
in which case they are not vertices of $\mathcal{G}_S$.
However it can be easily corrected by fixing, for every $i$, the line $l_i$ passing through $p_i$, 
with all the remaining points of $S'$ on one side and avoiding all points of $S$ other than $p_i$.
For every $1\le i < j \le n$, the point $q_{i,j}$ can then be defined as the intersection of the lines 
$l_i$ and $l_j$.
}
Thus, whenever $X$ is the complement of a set of finitely many points in the plane, 
$\lambda(X)$ can be bounded in terms of $\chi(X)$. 
This implies, by Theorem~\ref{thm:MatVal}, that the value of $\gamma(X)$ can be bounded in terms of $\chi(X)$,
settling the conjecture of Matou\v{s}ek and Valtr in the special case when $X$ is the complement of a finite set of points.

In this paper, we strengthen the result of Lawrence and Morris~\cite{LawrenceMorris09} 
and settle the conjecture for every planar set $X$.

\begin{theorem}
\label{thm:chromclique}
Any set $X\subseteq \Rbb^2$ with $\chi(X)=\chi < \infty$ satisfies
\[
\gamma(X) \le O\bigl(2^{2^{2^{\chi}+2}} \cdot \chi^3\bigr).
\]
\end{theorem}

We prove Theorem~\ref{thm:chromclique} in Section~\ref{sec:mainres}. 
In Section~\ref{sec:highdim}, we show that for every dimension $d$, $\lambda(X)$ 
can be bounded in terms of $\chi(X)$ for all sets $X \subset \Rbb^d$.
This answers Question 6 of Lawrence and Morris~\cite{LawrenceMorris09}.

A set $X$ is \emph{star-shaped} if $X$ contains a point $x \in X$ that sees every other point of $X$.
In Sections~\ref{sec:dim6} and~\ref{sec:dim5} we show that $\chi$ and $\gamma$ can be separated in dimensions $5$ and more, even for closed star-shaped sets. 

\begin{theorem}
\label{thm:dim6}
For every positive integer $g$ there exist star-shaped sets
\begin{enumerate}[{\rm 1)}]
\item \label{part:dim6a} $X \subset \Rbb^6$ satisfying $\chi(X)=2$ and $\gamma(X) \ge g$, and
\item \label{part:dim6b} $X_c \subset \Rbb^6$ that is closed and satisfies $\chi(X_c) \le 4$ and $\gamma(X_c) \ge g$.
\end{enumerate}
\end{theorem}

\begin{theorem}
\label{thm:dim5}
For every positive integer $g$ there exist star-shaped sets
\begin{enumerate}[{\rm 1)}]
\item \label{part:dim5a} $X \subset \Rbb^5$ satisfying $\chi(X)=2$ and $\gamma(X) \ge g$, and
\item \label{part:dim5b} $X_c \subset \Rbb^5$ that is closed and satisfies $\chi(X_c) \le 6$ and $\gamma(X_c) \ge g$.
\end{enumerate}
\end{theorem}

\begin{problem}
\label{prob:sep34}
Does there exist a function $f$ such that
$\gamma(X) \le f(\chi(X))$
\begin{enumerate}[{\rm 1)}]
\item
for every set $X \subseteq \Rbb^3$?
\item
for every set $X \subseteq \Rbb^4$?
\end{enumerate}

\end{problem}

All logarithms in this paper are binary.
The \emph{tower function} $T_l(k)$ is defined recursively as $T_0(k)=k$ and $T_h(k) = 2^{T_{h-1}(k)}$.
Its inverse is the iterated logarithm $\log^{(l)}(n)$, that is, $\log^{(0)}(n)=n$ and $\log^{(l)}(n) = \log(\log^{(l-1)}(n))$.
We use the notation $\overline{xy}$ for the straight line segment between points $x$ and $y$.


\section{Proof of Theorem~\ref{thm:chromclique}}
\label{sec:mainres}



In this and the next section we will use the following observation about one-point holes.
\begin{observation}
\label{obs:hole}
Let $q$ be a one-point hole in a set $X \subseteq \mathbb{R}^d$ with $\omega(X)<\infty$.
For every vector $x \in \mathbb{R}^d$ there is an $\varepsilon>0$ such that the open segment
between $q$ and $q+\varepsilon x$ is contained in $X$.
\end{observation}
\begin{proof}
For contradiction, suppose that there is a vector $x$ such that the open segment $s_{\varepsilon}$ between $q$ and $q+\varepsilon x$ is not
contained in $X$ for any $\varepsilon>0$. Now either the whole segment $s_{\varepsilon}$ is contained in
$\mathbb{R}^d \setminus X$ for some $\varepsilon>0$, or there is an infinite sequence of points $q_n\in X \cap s_{1}$ converging to $q$
such that for every $n$, there is at least one point of $\mathbb{R}^d \setminus X$ between $q_n$ and $q_{n+1}$.
In the first case, the hole is no longer a path-connected component of $\mathbb{R}^d \setminus X$.
In the second case, the sequence $q_n$ forms a clique of infinite size. Either way, we get a contradiction.
\end{proof}

Points $q_1, q_2, \dots , q_n$ are in \emph{clockwise convex position} if they form the vertices of a convex polygon
in the clockwise order. In particular, they are in general position and, for every $i<j$, the points
$q_{j+1}, q_{j+2}, \dots, q_{n}$ lie in the same half-plane of the line determined by $q_i$ and $q_j$.

To derive Theorem~\ref{thm:chromclique} from Theorem~\ref{thm:MatVal}, we need to show that the number of one-point
holes is bounded from above by a function of the chromatic number. 
We first show this in the special case
of one-point holes in convex position.

\begin{lemma}
\label{lem:specholesinplane}
Let $X \subseteq \Rbb^2$ be a set with $n$ one-point holes $q_1, q_2, \dots , q_n$
in clockwise convex position. Then
\[
n \le T_2(\chi(X)).
\]
\end{lemma}

\begin{proof}
For a pair of indices $i,j$, with $1 \le i < j \le n$, let $l(i, j)$ be the line containing
the point $q_{i}$ and passing in a small positive distance $d(i, j)$ (which we specify later) from $q_{j}$
in the direction from $q_{j}$ such that $q_{j}$ lies in the same half-plane 
as the points $\{q_{j+1}, q_{j+2}, \dots q_n\}$.

For a triple of indices $i,j,k$, with $1 \le i < j < k \le n$, let
$p(i, j, k) \mathrel{\mathop:}= l(i, j) \cap l(j, k)$, which is a point near $q_{j}$.

The distances $d(i, j)$ will be set to satisfy the following conditions (see Figure~\ref{fig:lines}).
\begin{enumerate}
\item[(i)] The lines $l(i, j)$ are in general position.
\item[(ii)] Each $l(i, j)$ contains exactly one of the holes (which is $q_{i}$).
\item[(iii)] For every pair $(i,j)$ the point $q_{j}$ and all the points $p(j, k, l)$ lie in the same half-plane
determined by $l(i, j)$ as the points $\{q_{j+1}, q_{j+2}, \dots q_n\}$.
\item[(iv)] Every point $p(i, j, k)$ is in $X$ and is closer to $q_j$ than $q_{i}$. 
\end{enumerate}

\myfig{lines}{1}{An example of four one-point holes in clockwise convex position and the corresponding lines.}

Since the one-point holes are in clockwise convex position,
there is an $\varepsilon$ such that whenever all $d(i, j)$ are greater than $0$ and at most $\varepsilon$,
then all the conditions are satisfied except possibly the first part of condition (iv). 

We start by setting the distances $d(i, n)$ to this $\varepsilon$, which allows us to place the lines $l(i,n)$.
Each distance $d(i, j)$, where $j<n$, is
set when all the lines $l(j, k)$ are already placed. Because $l(j, k)$ avoids $q_{i}$ and
by Observation~\ref{obs:hole}, there is $\varepsilon_{k}$ such that if $d(i, j) \le \varepsilon_{k}$, then
$p(i, j, k)$ lies in $X$. It is now enough to take $d(i, j)$ as the minimum of $\varepsilon$ and all the
$\varepsilon_{k}$.

Consider the graph with vertex set $\binom{[n]}{3}$ and edges between vertices
$\{i, j, k\}$ and $\{j, k, l\}$ for every $1 \le i < j < k < l \le n$.
This graph is called \emph{the shift graph $S(n,3)$} and its chromatic number is known to be
at least $\log^{(2)}(n)$~\cite{HellNesetril}.
We color the vertex $\{i, j, k\}$ with the color of the point
$p(i, j, k)$ in a fixed proper coloring $c$ of $X$ by $\chi(X)$ colors.

Assume that $n > T_2(\chi(X))$.
Then there are two points $s\mathrel{\mathop:}=p(i, j, k)$ and $t\mathrel{\mathop:}=p(j, k, l)$ with $c(s)=c(t)$.
Both $s$ and $t$ lie on the line $l(j, k)$, which also contains the point $q_{j}$.
To show that $q_{j}$ lies between $s$ and $t$ we use the fact that $s$
is the intersection of the lines $l(j, k)$ and $l(i, j)$.
The points $q_{j}$ and $t$ lie in the same direction from $l(i, j)$, but
$q_{j}$ is closer to $s$. Thus $s$ and $t$ are connected by an edge in the
invisibility graph of $X$, a contradiction.
\end{proof}

The following lemma is a slight modification of Exercise 3.1.3 from~\cite{MatousekLDG02}.
\begin{lemma}
\label{lem:selectinplane}
Any set $P \subset \Rbb^2$ of $m\cdot 2^{4n}$ points contains either $m$ points lying on a line
or $n+1$ points in convex position.
\end{lemma}

\begin{proof}
First we show that among any $mp^2$ points, we can either find $m$ points lying
on a line or $p$ points in general position.
Take a set $Q$ with no $m$ points on a line and no $p$ points in general position. Consider an arbitrary
inclusion-wise maximal subset $S \subseteq Q$ of points in general position. Each point from $Q\setminus S$
must thus lie on a line determined by some pair of points of $S$. There are fewer than
$p^2$ such lines and each contains fewer than $m-2$ points from $P\setminus S$.
Hence the total number of points of $Q$ is less than $mp^2$.

By the Erd\H{o}s--Szekeres theorem~\cite{ErdosSzekeres}, any set of $4^n$ points
in general position contains $n+1$ points in convex position.

Combining these two results we obtain that any set of $m\cdot 2^{4n}$ points in $\Rbb^2$
contains either $m$ points lying on a line or $2^{2n}$ points in general position and hence
$n+1$ points in convex position.
\end{proof}

\begin{lemma}
\label{lem:any1holesinplane}
Any set $X \subseteq \Rbb^2$ with $\lambda(X)=\lambda<\infty$ and $\chi(X)=\chi$ satisfies
\[
\lambda < \chi \cdot 2^{4 T_2(\chi)}.
\]
\end{lemma}

\begin{proof}
By Lemma~\ref{lem:selectinplane}, any set $P$ of
$\chi \cdot 2^{4 T_2(\chi)}$ points in $\Rbb^2$ contains either $\chi$ points on a line
or $T_2(\chi)+1$ points in general position.
If $X$ contains a line with $\chi$ one-point holes then $X$ cannot be colored with $\chi$ colors.
In the second case we get a contradiction by Lemma~\ref{lem:specholesinplane}.
\end{proof}


\begin{proof}[{Proof of Theorem~\ref{thm:chromclique}}]
By Lemma~\ref{lem:any1holesinplane}, the number of one-point holes in $X$ satisfies
$\lambda \le \chi \cdot 2^{4 T_2(\chi)}$. Thus by Theorem~\ref{thm:MatVal}
\begin{align*}
\gamma(X) &\le O(\omega^4 + \lambda \omega^2) \\
&\le O(\chi^4 + \chi \cdot 2^{4 T_{2}(\chi)} \cdot \chi^2).
\end{align*}
\qedhere
\end{proof}


\section{Generalizations to higher dimensions}
\label{sec:highdim}
The proof of Theorem~\ref{thm:chromclique} has two main components, Theorem~\ref{thm:MatVal}
and Lemma~\ref{lem:any1holesinplane}.
In particular, $1$-dimensional and $2$-dimensional holes are handled by Theorem~\ref{thm:MatVal} and
$0$-dimensional holes by Lemma~\ref{lem:any1holesinplane}.

We can generalize Lemma~\ref{lem:any1holesinplane} to any dimension $d$ (see Lemma~\ref{lem:any1holes}), but it is unclear whether Theorem~\ref{thm:MatVal} can be generalized to dimension $3$. We know, however, that it cannot be generalized to dimension $5$ or more, due to Theorem~\ref{thm:dim5}.

The proof of Theorem~\ref{thm:MatVal} is composed of three main steps. The set $X$ is reduced to a closed polygonal
set, which is then decomposed into pseudotrapezoids that behave essentially as star-shaped sets.
The last step uses the bound $\gamma(X) \le 2 \omega(X)$ for closed star-shaped sets $X$ proved by
Breen and Kay~\cite[Corollary 3]{BreenKay76}.

A bounded planar set $X$ is \emph{polygonal} if its boundary is composed of finitely many points and open segments, 
each of which is either contained in $X$ or disjoint from $X$.
More generally, we say that a bounded set $X \subset \Rbb^d$ is \emph{polyhedral} if its boundary is composed of finitely
many lower-dimensional polytopes, each of which is either contained in $X$ or disjoint from $X$.
Given a planar set $X$ and a finite set $P \subseteq X$ of points, the \emph{convex hull $\conv_X(P)$ of $P$ relative to $X$}
is the minimum set $Y \subseteq X$ such that $P \subseteq Y$ and every segment $\overline{xy} \subseteq X$ with $x,y \in Y$ 
satisfies $\overline{xy} \subseteq Y$.
Unlike Lemma 5.2 in~\cite{MatousekValtr99}, for $X \subset \Rbb^3$ with $\omega(X)< \infty$ and a finite set $P\subset X$,
the convex hull $\conv_X(P)$
of $P$ relative to $X$ is not always polyhedral. An example of such $X$ is the union of the cylinder
$\{(x,y,z): x^2+y^2 \le 1, |z| \le 1\}$ and two squares $\{(x,y,1): |x|,|y| \le 1\}$ and
$\{(x,y,-1): |x|,|y| \le 1\}$. We choose $P$ to be the set of the eight points with each coordinate
either $-1$ or $1$. Then $\omega(X)=3$ and $\conv_X(P) = X$, which is not a polyhedral set.

As a particular consequence of this fact, a different classification of the holes is needed. For example, the moment curve $\gamma(t)=(t,t^2,t^3)$ in $X=\Rbb^3 \setminus \gamma$ cannot be treated as a one-dimensional hole of $X$, since $\chi(X)=\infty$ (which follows from Lemma~\ref{lem:specholes}) and $\omega(X)=3$ (we leave this as an exercise for the interested reader).

We suggest the following classification of points of the complement of $X$. A point $x \in \Rbb^d \setminus X$
is \emph{$k$-dense} if there is an affine $k$-dimensional subspace $Y$ such that the intersection of every open neighborhood of $x$ with $(\Rbb^d \setminus X)\cap Y$ has positive $k$-dimensional
Lebesgue measure. A point $x \in \Rbb^d \setminus X$ is \emph{$k$-sparse} if it is not $k$-dense. Points that are $d$-dense or $d$-sparse are simply called \emph{dense} or \emph{sparse}, respectively.

Lemma~\ref{lem:specholes} together with Lemma~\ref{lem:selectpoints} bound the maximum number of $k$-sparse points in general position in any $k$-dimensional affine subspace in terms of $\chi(X)$.

Regarding the other parts of the proof of Theorem~\ref{thm:MatVal}, there is no obvious way of generalizing the pseudotrapezoid decomposition in $\Rbb^3$, and we are also missing a generalization of Breen and Kay's result to dimension $3$.

\begin{problem}
Is $\gamma(X)$ or $\chi(X)$ bounded in terms of $\omega(X)$ for closed star-shaped sets $X\subset \Rbb^3$?
\end{problem}

The answer to this question in dimension $4$ is negative. Kojman, Perles and Shelah~\cite{KPS90} constructed a star-shaped set $X \in \mathbb R^4$ with $\omega(X)=2$ and $\chi(X)=2^{\aleph_0}$. We describe a polygonal analogue of this construction. Let $P_n$ be a $4$-dimensional cyclic polytope with $n$ vertices. Note that every pair of vertices forms an edge of $P_n$, therefore the $1$-skeleton of $P_n$ is a straight-line embedding of the complete graph $K_n$. Let $H\subset K_n$ be a triangle-free graph with chromatic number $\Omega(\log n)$; for example, $H$ may be the \emph{shift graph $S(\sqrt{n},2)$}~\cite{HellNesetril}.
Choose a point $x_e$ from each edge $e$ of $P_n$ corresponding to an edge of $H$. By removing all points $x_e$ from $P_n$, we get a star-shaped set $X'$ with $\omega(X')=2$ and $\chi(X')=\Omega(\log n)$. To get a closed set, we remove a small and sufficiently flat wedge-shaped neighborhood of each point $x_e$. The resulting set $X$ still has $\omega(X)=2$ and $\chi(X)=\Omega(\log n)$.

%
%
%
%
%
%



\subsection{Sparse points in higher dimensions}

We say that an ordered set $(q_1, q_2, \dots, q_n)$ of $n$ points in $\Rbb^d$ is in \emph{same-side position\/}
if they are in general position and for every $d$-tuple $q_{i_1} q_{i_2}, \dots , q_{i_d}$
where $i_1<i_2<\dots<i_d$, the points $q_{i_d+1}, q_{i_d+2}, \dots, q_n$ lie in
a common open half-space determined by the hyperplane spanned by $q_{i_1}, q_{i_2}, \dots , q_{i_d}$.
An unordered set is in \emph{same-side position} if some ordering of its points is in same-side 
position.

Notice that every set of points in clockwise convex position in $\mathbb{R}^2$ is in same-side position.
On the other hand, every set of four or five points in $\mathbb{R}^2$ is in same-side position, 
while some of these sets are not in convex position.
But if both $(q_1, q_2, \dots, q_n)$ and  $(q_n, q_{n-1}, \dots, q_1)$ are in same-side position in $\mathbb{R}^d$, then the set $\{q_1, q_2, \dots, q_n\}$ is in convex position. 
This is easy to verify for $n=d+1$ and the case for general $n$ follows by Carath\'{e}odory's 
theorem (see for example~\cite{MatousekLDG02}).


\begin{lemma}
\label{lem:specholes}
Let $X \subseteq \Rbb^d$ be a set with $n$ sparse points $q_1, q_2, \dots , q_n$
in same-side position. Then
\[
n \le T_{2d-2}(\chi(X)).
\]
\end{lemma}

\begin{proof}
For every $d$-element set of indices $1 \le i_1 < i_2 < \dots < i_d \le n$, we define a hyperplane $h(i_1, i_2, \dots, i_d)$ passing through the points $q_{i_1}, q_{i_2}, \dots , q_{i_{d-1}}$ and near the point $q_{i_d}$ at distance $s_{i_1, i_2, \dots, i_d}>0$, which will be determined later.
The hyperplane is selected so that the points $q_{i_d}, q_{i_d+1}, \dots q_n$ are all in one half-space
determined by $h(i_1, i_2, \dots, i_d)$.

For a set of indices $1 \le i_1 < i_2 < \dots < i_{2d-2} \le n$ let
\[
l(i_1, i_2, \dots, i_{2d-2}) \mathrel{\mathop:}= h(i_1, i_2, \dots, i_d) \cap h(i_2, i_3, \dots, i_{d+1}) \cap \dots \cap h(i_{d-1}, i_d, \dots, i_{2d-2})
\]
and for a set of indices $1 \le i_1 < i_2 < \dots < i_{2d-1} \le n$ let
\[
p(i_1, i_2, \dots, i_{2d-1}) \mathrel{\mathop:}= h(i_1, i_2, \dots, i_d) \cap h(i_2, i_3, \dots, i_{d+1}) \cap \dots \cap h(i_d, i_{d+1}, \dots, i_{2d-1}).
\]

Observe that $l(i_1, i_2, \dots, i_{2d-2})$ is a line containing the point $q_{i_{d-1}}$ and passing near
$q_{i_d}$. Also observe that
$p(i_1, i_2, \dots, i_{2d-1})$ is a point in the intersection of the lines $l(i_1, i_2, \dots, i_{2d-2})$
and $l(i_2, i_3, \dots, i_{2d-1})$ and that it lies near $q_{i_d}$.

Let $B(q_i,\varepsilon)$ be the closed ball of radius $\varepsilon$ centered in $q_i$.
We pick an $\varepsilon > 0$ such that the measure of each $B(q_i,\varepsilon) \setminus X$ is zero
and that no $d+1$ balls $B(q_i,\varepsilon)$ can be intersected by a hyperplane. This is possible
due to the sparsity and the general position of the points $q_1, q_2, \dots, q_n$. Further we pick
a $\delta \in (0,\varepsilon)$ such that for all choices of all $s_{j_1, j_2, \dots, j_d} \in (0, \delta)$,
each point $p(i_1, i_2, \dots, i_{2d-1})$ lies inside $B(q_{i_d},\varepsilon)$.


We will now select the distances $s_{j_1, j_2, \dots, j_d}$ so that each point $p(i_1, i_2, \dots, i_{2d-1})$
lies inside $B(q_{i_d},\varepsilon) \cap X$.

The distances $s_{i_1, i_2, \dots, i_d}$ will be set in the order of nonincreasing $i_d$. Thus the distance
$s_{i_1, i_2, \dots, i_d}$ will be set when all the hyperplanes $h(i'_1, i'_2, \dots, i'_d)$ with $i'_d>i_d$ are already fixed.

For every $j$ such that $0 \le j \le d-2$ and every $i_1 < i_2 < \dots < i_{d+j}$, let $a(i_1, i_2, \dots, i_{d+j})$
be the affine subspace $h(i_1, i_2, \dots, i_d) \cap h(i_2, i_3, \dots, i_{d+1}) \cap \dots \cap h(i_{j-1}, i_j, \dots, i_{d+j}))$.

Each distance $s_{i_1, i_2, \dots, i_d}$ is chosen so that the hyperplane $h(i_1, i_2, \dots, i_d)$ satisfies:
\begin{enumerate}
\item \label{it:lem9:2}
For every $i_{d+1} < i_{d+2} < \dots < i_{2d-1}$ where $i_{d+1} > i_d$, the point $p(i_1, i_2, \dots, i_{2d-1})$
lies inside $X$.
\item \label{it:lem9:3}
For every $j$ such that $0\le j \le d-2$ and every $i_{d+1} < i_{d+2} < \dots < i_{d+j}$ where $i_{d+1} > i_d$, the set
$(B(q_{i_d},\varepsilon) \cap a(i_1, i_2, \dots, i_{d+j})) \setminus X$ has $(d-j-1)$-dimensional measure zero.
\end{enumerate}

The point $p(i_1, i_2, \dots, i_{2d-1})$ is the intersection of $h(i_1, i_2, \dots, i_d)$
with the line $l(i_2,\allowbreak i_3, \dots,\allowbreak i_{2d-1})$.
Because the hyperplanes $h(i'_1, i'_2, \dots, i'_d)$ with $i'_d>i_d$
satisfy condition~\ref{it:lem9:3}, the set of numbers $s$ from the interval $(0,\delta)$
that violate condition~\ref{it:lem9:2}~or~\ref{it:lem9:3} has $1$-dimensional measure zero.
We can thus find $s_{i_1, i_2, \dots, i_d}$ satisfying both conditions.



The \emph{shift graph $S(n,k)$} is the graph whose vertex set is $\binom{[n]}{k}$
and whose edges are between vertices $\{i_1, i_2, \dots, i_k\}$ and $\{i_2, \dots, i_k, i_{k+1}\}$ 
for every $1 \le i_1 < i_2 < \dots < i_k < i_{k+1} \le n$.
The chromatic number of $S(n,k)$ is known to be at least $\log^{(k-1)}(n)$~\cite{HellNesetril}.

As in the planar case, we can apply a coloring of $X$ to the vertices of the shift graph $S(n,2d-1)$.
For every $1 \le i_1 < \dots < i_{2d-1} \le n$, the vertex $\{i_1, i_2, \dots, i_{2d-1}\}$ is colored 
by the color of the point $p(i_1, i_2, \dots, i_{2d-1})$ in a fixed proper coloring $c$ of $X$ by $\chi(X)$ colors.

If $n > T_{2d-2}(\chi(X))$, then $\chi(X) < \chi(S(n,2d-1))$ and there are points
$s\mathrel{\mathop:}=p(i_1, i_2, \dots, i_{2d-1})$ and $t\mathrel{\mathop:}=p(i_2, i_3, \dots, i_{2d})$ with $c(s)=c(t)$.
Both of them lie on the line $l(i_2, i_3, \dots, i_{2d-1})$, which also contains the
point $q_{i_d}$. To show that $q_{i_d}$ lies between $s$ and $t$ we observe that $s$
is the intersection of the line $l(i_2, i_3, \dots, i_{2d-1})$ and the hyperplane
$h(i_1, i_2, \dots, i_{d})$. The point $q_{i_d}$ is near this hyperplane and on
the same side as all the points $q_{i_d + 1}, q_{i_d + 2}, \dots, q_{n}$. Because
$t$ is near $q_{i_{d+1}}$, it lies in the same direction from $s$ as $q_{i_d}$, but
$q_{i_d}$ is closer to $s$. Thus $s$ and $t$ are connected by an edge in the
invisibility graph of $X$, a contradiction.
\end{proof}

Let $g_1(n)\mathrel{\mathop:}=n$. For any $n,d>1$, let $g_d(n)$ be the smallest number $g$ such that any set
$P \subset \Rbb^d$ of $g$ points contains either $g_{d-1}(n)$ points lying in a hyperplane
or an $n$-tuple $(q_1, q_2, \dots, q_n) \subset P$ in same-side position. 
In other words, any set of $g$ points in $\Rbb^d$ contains either
$n$ points on a line or there is an affine subspace of $\Rbb^d$ with $n$ points of $P$
in same-side position.
If no such number $g$ exists, then we say that $g_d(n)$ is infinite.

In the special case of $d=2$, Lemma~\ref{lem:selectinplane} proves that $g_2(n) \le n2^{4n}$.

The following lemma is a slight modification of Exercise 5.4.3 from~\cite{MatousekLDG02}.
\begin{lemma}  
\label{lem:selectpoints}
The value $g_d(n)$ is finite for all $n,d>0$.

\end{lemma}

\begin{proof}
The \emph{orientation\/} of a $(d+1)$-tuple of points $(p_1, p_2, \dots, p_{d+1})$ in $\Rbb^d$ is the sign
of the determinant of the matrix $A$ whose columns are the $d$ vectors $p_2-p_1, p_3-p_1, \dots, p_{d+1}-p_1$.
The orientation is equal to $0$ if and only if the points $(p_1, p_2, \dots, p_{d+1})$ lie in a hyperplane.
Otherwise it is equal to $+1$ or $-1$ and determines on which side of the hyperplane spanned
by ${p_1, p_2, \dots, p_{d}}$ the point $p_{d+1}$ lies.

Let $\mathrm{R}_k(l_1, l_2, \dots, l_c; c)$ be the Ramsey number denoting the smallest number $r$ such that
if the hyperedges of a complete $k$-uniform hypergraph on $r$ vertices are colored with $c$
colors, then for some color $i$, the hypergraph contains a complete sub-hypergraph on $l_i$ vertices
whose hyperedges are all of color $i$.

We will prove the lemma by showing that $g_d(n) \le \mathrm{R}_{d+1}(n, g_{d-1}(n), n; 3)$ for all $d,n>1$.
We take any set $P$ of at least $\mathrm{R}_{d+1}(n, g_{d-1}(n), n; 3)$ points.
Then we color the $(d+1)$-tuples of points from $P$ with the colors $-1, 0, +1$ determined by
the orientation of the $(d+1)$-tuple.

If there are $g_{d-1}(n)$ points such that each $(d+1)$-tuple of them has orientation $0$ and
thus lies in a hyperplane, then all $g_{d-1}(n)$ points lie in a common hyperplane.

Otherwise there are, without loss of generality, $n$ points $q_1, q_2, \dots, q_n$ such that
each $(d+1)$-tuple of them has orientation $+1$. Then for every $d$-tuple
$(q_{i_1}, q_{i_2}, \dots , q_{i_d})$, the points $q_{i_d+1}, q_{i_d+2}, \dots, q_n$ lie in one half-space
determined by the hyperplane spanned by $\{q_{i_1}, q_{i_2}, \dots , q_{i_d}\}$.
\end{proof}

We note that by a straightforward extension of a recent result by Suk~\cite{Suk14}, it is possible to bound $g_d(n)$ from above by $T_{d-1}(O(n))$.

\begin{lemma}
\label{lem:any1holes}
Any set $X \subseteq \Rbb^d$ with $\lambda(X)=\lambda<\infty$ and $\chi(X)=\chi$ satisfies
\[
\lambda < g_d(T_{2d-2}(\chi)).
\]
\end{lemma}

\begin{proof}
We proceed by induction on $d$.

In $\Rbb^1$, any set $X$ with $\lambda$ one-point holes needs at least $\lambda + 1$
colors. Thus $\lambda < \chi(X) = g_1(T_{0}(\chi))$.

Suppose that $X$ has a set $P$ of $g_d(T_{2d-2}(\chi))$ one-point holes.
By Lemma~\ref{lem:specholes}, there is no set $P'\subseteq P$ of $T_{2d-2}(\chi)$ one-point
holes in same-side position.
Thus by the definition of $g_d$, there are
$g_{d-1}(T_{2d-2}(\chi)) \ge g_{d-1}(T_{2d-4}(\chi))$ one-point holes
in one hyperplane and we obtain a contradiction with the induction hypothesis.
\end{proof}


\section{Constructions in dimension 6}
\label{sec:dim6}
\subsection{Set with chromatic number 2}

We prove part~\ref{part:dim6a} of Theorem~\ref{thm:dim6}.

Let $P_n$ be the \emph{cyclic polytope} on $n\ge 7$ vertices (see for example~\cite{MatousekLDG02}) and let $V_n$ be its set
of vertices. The cyclic polytope is an example of a neighborly polytope, which means that the convex hull of every triple of points from $V_n$ is a triangular face of $P_n$.

\begin{lemma}
\label{lem:dim6-dir-triangle}
Let $n\ge 7$, let $V$ be a set of $n$ elements and let $K_n$ be the complete graph with vertex set $V$. Let
$k\mathrel{\mathop:}=\lceil 2\log(n) + 2 \rceil$.
The edges of $K_n$ can be oriented so that every set $V' \subseteq V$
of size at least $k$ contains a directed triangle.
\end{lemma}

\begin{proof}
For brevity, we call a set $V'$ \emph{good} if it contains a directed triangle.

We orient the edges randomly and show that with positive probability, every set $V' \subseteq V$ of size
at least $k$ is good.

First, we will bound the probability $b_k$ that a given set $V'$ of $k$ vertices is bad.
If there exists a directed cycle on $V'$ of length greater than $3$, then one of the two cycles created
by adding an arbitrary diagonal to the cycle is again directed. By induction there exists a directed triangle
on $V'$.

There are $2^{k(k-1)/2}$ possible orientations of the edges of a complete graph on $k$ vertices, out of which
$k!$ are acyclic. Thus
\[
b_k = \frac{k!}{2^{k(k-1)/2}} = k! 2^{-k^2/2 + k/2}.
\]
The probability that some $k$-tuple $V'$ of vertices is bad is thus at most
\begin{align*}
\binom{n}{k} b_k 
\le  \frac{n^{k}}{k!} b_k  =
2^{k\log(n) - k^2/2 + k/2} =
2^{k(\log(n) - k/2 + 1/2)} 
\le 2^{k(-1/2)} < 1.
\end{align*}
\end{proof}

We fix the orientation of the edges of $P_n$ in which every $k$-tuple of vertices is good. Every oriented edge has two distinguished endpoints, the \emph{head} and the \emph{tail}, and it is oriented towards its head.
A triangular face of $P_n$ has \emph{directed boundary} if the three edges of the face form a directed cycle.
The set $X$ is constructed by puncturing a one-point hole in the barycenter of each triangular face of $P_n$
with directed boundary.

The vertices of $P_n$ are colored black. The edges are cut in thirds.
In every edge, the interior of the middle third together
with the point at one third closer to the head of the edge is colored white. 
The rest of the edge is colored black.
The interior of every triangular face $F$ with directed boundary is colored like on Figure~\ref{fig:R6}; that is, the points on each ray originating in the barycenter of $F$ have the same color, which is determined by the color of the point where the ray intersects the boundary of $F$.
The rest of $X$ is colored black.

\myfig{R6}{1}{A coloring of a face with directed boundary. The lines determined by pairs $(p_1,p_4)$, $(p_2,p_5)$ and $(p_3,p_6)$ intersect in the barycenter and split the triangle into monochromatic regions. Full lines and gray regions represent black color, the rest is white.}

Every edge of the invisibility graph of $X$ joins
two points lying in the same triangular face with directed boundary.
The coloring is proper on each of these faces and thus the $2$-coloring of the whole set $X$ is proper.

If a convex set $C$ contains at least $k$ vertices of $X$, then it contains a triangular
face with directed boundary and thus $C$ contains a one-point hole. Therefore
\[
\gamma(X) \ge \frac{n}{2\log(n)+2}.
\]

\begin{remark}
An alternative construction can be obtained from constructions of partial 
Steiner triple systems with small independence number.

A \emph{partial Steiner triple system (partial STS)} on a set $V$ of size $n$ 
is a collection $S$ of triples of elements of $V$ such that for every distinct 
$s_1, s_2 \in S$, we have $|s_1 \cap s_2| \le 1$.
Brandes, Phelps and R\"{o}dl~\cite{BPR82} gave a probabilistic construction of
a partial STS on a set $V$ of size $n$ such that every subset of at least 
$c\sqrt{n \log(n)}$ elements of $V$ contains a triple from the STS, where $c$ 
is an absolute constant.
Phelps and R\"{o}dl~\cite{PhelpsRodl86} proved that this construction is the best 
possible, up to the value of the constant $c$.

We can then construct a set $X'$ from $P_n$ by making a one-point hole in 
the barycenters of the triangular faces corresponding to the triples in this 
partial STS on the set $V_n$.
Then $X'$ has chromatic number $2$ and $\gamma(X') \ge  \sqrt{n/\log(n)}/c$.

Such a partial STS also provides an alternative  construction of a closed set 
with chromatic number at most $4$ and arbitrarily large convexity number.
\end{remark}

\subsection{Closed set with chromatic number $4$}
\label{subsec:dim6closed}
For every $\gamma$, we construct a closed set $X_c$ and then show that it is star-shaped,
$\gamma(X_c) = \gamma$ and $\chi(X_c)=4$.

\subsubsection{The construction}

Consider a graph $G$ with colored edges.
A triangle in $G$ with three different colors of its edges is a \emph{rainbow triangle}.
Let $k(n) \mathrel{\mathop:}= \lfloor 12\log(n) \rfloor$.
We prove a slightly stronger version of the classical exponential lower bound on the three-color Ramsey numbers of complete graphs.

\begin{lemma}
\label{lem:dim5-rainbow-triangle}
For every $n \ge n_0$ where $n_0$ is an absolute constant, 
there is a coloring $\phi: E(K_n) \rightarrow \{1,2,3\}$ of the edges of the complete graph $K_n$
such that the following holds.
Every set $V' \subseteq V_n$ of vertices of $K_n$ of size
$|V'| = k(n)$ contains a triple $v_1, v_2, v_3 \in V'$ that induces a rainbow triangle.
\end{lemma}
\begin{proof}

Every edge is assigned one of the colors $1$, $2$ and $3$ uniformly and independently at random.
For brevity, we call a triple of vertices that does not induce a rainbow triangle
a \emph{bad triple}.

First, we will bound the probability $b'_k$ that a set $V'$
of vertices $1,2,\dots k$ contains only bad triples of vertices.
A triple $T$ is bad with probability $7/9$. 
Let $E_{T}$ be the event that the triple $T$ is bad.
Let $\mathcal T$ be the set of triples of the form $\{2i-1,2i,j\}$ for some $i\le k/2$ and $j>2i$.
Thus $|\mathcal T| \ge k(k-2)/4$.
Notice that the events $E_T$ for triples $T \in \mathcal T$ are mutually independent. 
This is because
each such $T=\{2i-1,2i,j\}$ contains two edges $\{2i-1,j\}$ and $\{2i,j\}$ not present in any other
$T'\in \mathcal T$, and the conditional probability of $E_T$ if the color of $\{2i-1,2i\}$ is fixed is still $7/9$.
The value $n_0$ is chosen so that $k(n)-2 \ge \log(n) \cdot 4/\log(9/7)$
for every $n \ge n_0$.
With $k \mathrel{\mathop:}= k(n) = \lfloor 12 \log(n) \rfloor$, we have
\[
b'_k \le \left( \frac79 \right)^{|\mathcal T|} \le 2^{-k(k-2) \log(9/7)/4}
\le 2^{-k \log(n)}.
\]
The probability that some $V'$ of size $k$ contains a bad triple is at most
\[
b'_k \cdot \binom{n}{k} <  b'_k n^{k} \le 2^{- k \log(n)} 2^{k \log(n)} = 1.
\]
\end{proof}

Let $P_n$ be the cyclic polytope with vertex set $V_n$ of size $n \ge n_0$.
We fix a coloring $\phi: E(K_n) \rightarrow \{1,2,3\}$
with a rainbow triangle on every $k(n)$-tuple of vertices,
which exists by Lemma~\ref{lem:dim5-rainbow-triangle}.
For a face $F$ of $P_n$, let $V(F)$ be the set of vertices of $F$.
A triangular face $S$ of $P_n$ is a \emph{rainbow triangular face} 
if $V(S)$ forms a rainbow triangle in $\phi$.

The set $X_c$ is a closed star-shaped set constructed from $P_n$
by cutting out some of its parts, which we now describe in detail.

The \emph{kernel} of a set $X\subset \mathbb{R}^d$ is the set of points of $X$ that ``see'' all the points of $X$.
That is, for every point $q$ in the kernel of $X$ and every $x \in X$, the segment $\overline{px}$
is fully contained in $X$.
Star-shaped sets are precisely the sets with non-empty kernel.

We consider the whole $P_n$ as a $6$-dimensional face of itself.
We call a face $F$ of $P_n$ an \emph{ordinary face} if it is a triangular non-rainbow face 
or a face of dimension $3$, $4$, $5$ or $6$.
For every ordinary face $F$ of $P_n$, we fix a point $q_F$ in the interior of $F$.
When constructing $X_c$,
we make sure that the point $q_F$ remains in the kernel of $X_c \cap F$.

For every triangular rainbow face $S$ with $V(S) = \{u, v, w\}$, we do the following.
We consider the triangle $T_S = b + 1/10 (S-b)$ where $b$ is the barycenter of $S$.
That is, $T_S$ is a homothetic copy of $S$ with the same barycenter.
Let $u'$, $v'$ and $w'$ be the vertices of $T_S$ corresponding to $u$, $v$ and $w$, respectively.
For every $x \in \{u,v,w\}$, let $e_x$ be the edge of $S$ that is not incident with $x$
and let $e'_x$ be the corresponding edge of $T_S$.
We take three hyperplanes $h_u$, $h_v$ and $h_w$ 
satisfying the following three conditions for every $x \in \{u,v,w\}$.

\begin{enumerate}
\item[(C1)]
\label{enum:C1}
The hyperplane $h_{x}$ contains the edge $e'_x$ of $T_S$.
\item[(C2)]
\label{enum:C2}
One open halfspace determined by $h_{x}$ contains $x$ and the interior of $T_S$, 
this halfspace is called the \emph{minor halfspace} $H_x^-$. 
The other open halfspace is the \emph{major halfspace} $H_x^+$ and contains 
the point $q_F$ for every ordinary face $F$,
and all the vertices of $P_n$ except for $x$.
\item[(C3)]
\label{enum:C3}
Let $e$ be an edge of $P_n$ incident to $x$ other than the two edges incident to $S$.
Then at most one quarter of $e$ is in $H_x^-$.
Note that all the edges not incident to $x$ are entirely in $H_x^+$
and that less than three quarters of each of the edges of $S$ incident to $x$ are in $H_x^-$.
\end{enumerate}

Let $C_S \mathrel{\mathop:}= H_u^- \cap H_v^- \cap H_w^-$.
Let $\mathcal{S}$ be the set of all triangular rainbow faces of $P_n$ 
and let $X_c \mathrel{\mathop:}= P_n \setminus \bigcup_{S \in \mathcal{S}}{C_S}$.
Clearly, the set $X_c$ is closed.

For every $S \in \mathcal{S}$, let $\Delta_S = P_n \cap C_S$.
That is, $\Delta_S$ is a $6$-dimensional simplex with part of its boundary missing.
Four facets of $\Delta_S$ are determined by four of the hyperplanes defining $P_n$
and the other three by $h_u$, $h_v$ and $h_w$.
We have $X_c=P_n \setminus \bigcup_{S \in \mathcal{S}}{\Delta_S}$.

\subsubsection{Properties of the constructed set}

First, we show that the sets $\Delta_S$ are pairwise disjoint 
and so the three facets of $\Delta_S$ defined by the hyperplanes $h_u$, $h_v$ and $h_w$ 
are also facets of $X_c$.

\begin{lemma}
\label{lem:disj-removals}
Let $S_1$ and $S_2$ be two different rainbow triangular faces of $P_n$.
Then the set $\Delta_{S_2}$ is contained in the major halfspace of each of 
the three hyperplanes defining $S_1$.
In particular, the sets $\Delta_{S_1}$ and $\Delta_{S_2}$ are disjoint.
\end{lemma}
\begin{proof}
Let $v \in V(S_2) \setminus V(S_1)$.
Let $h_v$ be the hyperplane defining $C_{S_2}$ and satisfying $v \in H_v^-$.
Let $u \in V(S_1)$ and let $h_u$ be the hyperplanes defining $C_{S_1}$ such that $u \in H_u^-$.
It is enough to show that $H_u^- \cap H_v^- \cap P_n = \emptyset$.

By~(C2), the set $H_u^- \cap P_n$ is a convex polytope whose vertex set contains only $u$ 
and one point from the interior of each edge incident to $u$.
The point $u$ and the edges not incident to $v$ lie in $H_v+$.
Since the edge $\overline{uv}$ is not an edge of $S_1$, 
its intersections with $H_u^-$ and $H_v^-$ are disjoint by~(C3).
All the vertices of the convex polytope $H_u^- \cap P_n$ are thus in $H_v^+$ 
and so the whole set $H_u^- \cap P_n$ is in $H_v^+$ as well.
\end{proof}

\begin{observation}
\label{obs:halfspaces}
Let $H_1^-$, $H_2^-$ and $H_3^-$ be three open halfspaces.
Let $x$ and $y$ be two points from the complement of $H_i^-$ for some $i \in \{1,2,3\}$.
Then the line segment $\overline{xy}$ does not intersect $H_1^- \cap H_2^- \cap H_3^-$.
\qed
\end{observation}

\begin{observation}
For every ordinary face $F$ of $P_n$, the set $F \cap X_c$ is star-shaped.
In addition, for every ordinary face $F' \supseteq F$,
the point $q_F$ lies in the kernel of $F' \cap X_c$.
\end{observation}

\begin{proof}
By (C2), we have $q_F \in F' \cap X_c$.
Let $x$ be a point from $F' \cap X_c$ and let $S$ be a triangular rainbow face of $P_n$.
Let $h_u$, $h_v$ and $h_w$ be the three hyperplanes that determine $C_S$.
We need to show that the segment $\overline{xq_F}$ does not intersect $C_S$.
Since $x \notin C_S$, 
the point $x$ lies in at most two of the minor halfspaces $H_u^-$, $H_v^-$ and $H_w^-$.
The point $q_F$ does not lie in any of these minor halfspaces 
and thus $\overline{xq_F}$ does not intersect $C_S = H_u^- \cap H_v^- \cap H_w^-$ by Observation~\ref{obs:halfspaces}.
\end{proof}

When a convex set $D \subseteq X_c$ contains at least $k(n)$ vertices of $P_n$, 
then it contains a rainbow triangular face $S$ by Lemma~\ref{lem:dim5-rainbow-triangle}.
Thus $D$ contains the triangle $T_S$ removed from $S$ and so $D$ is not a subset of $X_c$.
Therefore
\[
|D \cap V_n| < k(n) = \lfloor 12 \log n \rfloor \quad \text{and} \quad \gamma(X_c) \ge \frac{n}{12\log n}.
\]

\subsubsection{Coloring}
We describe a proper coloring $\zeta: X_c \rightarrow \{1,2,3,4\}$. 
That is, a coloring such that every segment with endpoints of the same color is contained in $X_c$.

All the vertices of $P_n$ get color $4$.
For every edge $e = \overline{uv}$ of $P_n$, 
we color the two closed segments of points at distance at most $\|u-v\|/4$ from $u$ or $v$ 
with color $4$. 
The remaining points on $e$ are colored by the color of  
the edge $\{u,v\}$ in the fixed coloring $\phi:E(K_n) \rightarrow \{1,2,3\}$.

For every triangular rainbow face $S$ of $P_n$ we do the following.
The vertices of $T_S$ get color $4$.
The interior of each of the three edges of $T_S$ gets the color of the middle segment 
of the corresponding edge of $S$. 
Then we color the interior of $S \setminus T_S$ as in Figure~\ref{fig:R6-closed}~a).
That is, for every $i \in \{1,2,3\}$, 
all the points in the convex hull of points already colored by color $i$ get color $i$.
All the other points of $S \cap X_c$ get color $4$.
Each of the three facets of the simplex $C_S \cap P_n$ 
that make part of the boundary of $X_c$ 
is colored by the color of the edge of $T_S$ it contains,
with the exception of the vertices of $T_S$ that have color $4$.
The points in the intersection of two or three facets get an arbitrary color of the colors on the facets they lie in.

\myfig{R6-closed}{1}{Examples of the coloring of the interior of triangular faces. 
a) A rainbow triangular face. b) An ordinary triangular face.}

\begin{definition}
Let $F$ be a star-shaped set with colored boundary.
Let $p$ be a point in the kernel of $F$.
The \emph{$p$-extension} of the coloring of the boundary 
is the following coloring of $F$.
For every point $x$ on the boundary of $F$,
all points on the open segment between $p$ and $x$ get the color of $x$.
The point $p$ gets an arbitrary color of those used on the boundary. 
\end{definition}

Now we color all the uncolored points of $X_c$.
Each such point lies in the intersection of $X_c$ 
and an ordinary face of $P_n$.
We start with non-rainbow triangular faces, 
then we color $3$-, $4$-, $5$- and $6$-dimensional faces, in this order.
For each such face $F$, 
$F \cap X_c$ is colored by the $q_F$-extension of the coloring of its boundary.
An example of a coloring of a triangular non-rainbow face is depicted in
Figure~\ref{fig:R6-closed}b).


\begin{observation}
\label{obs:extension}
Let $R$ be a star-shaped set and let $p$ be a point in the kernel of $R$.
Let $r \neq p$ be a point of $R$ and let $r'$ be a point on the ray emanating from $p$ and passing through $r$
lying further away from $p$ than $r$.
For every $s \in R$ if $\overline{sr'} \subseteq R$ then $\overline{sr} \subseteq R$.
\end{observation}

\begin{proof}
Refer to Figure~\ref{fig:dim6extProof}.
If $p$, $r$ and $s$ are collinear, the claim is trivial.
Otherwise, let $T$ be the triangle with vertices $p$, $r'$ and $s$.
Since $p$ is in the kernel of $R$, the intersection of $R$ and the plane containing $T$ is star-shaped.
The boundary of $T$ is a subset of $R$ and so $T \subseteq R$.
\end{proof}

\myfig{dim6extProof}{1}{An illustration to Observation~\ref{obs:extension}.}

A \emph{special point} is a point of $X_c$ that is not in the interior of $F \cap X_c$ 
for any ordinary face $F$ of $P_n$.
That is, special points are vertices of $P_n$, points on the edges of $P_n$,
and, for every rainbow triangular face $S$, points of $S \cap X_c$ and $\Delta_S \cap X_c$.
A point of $X_c$ that is not special is an \emph{ordinary point}.

\begin{lemma}
\label{lem:dim6specpoints}
Let $S$ be a triangular rainbow face with $V(S)=\{u,v,w\}$.
\begin{enumerate}[a)]
\item
Let $i \in \{1,2,3\}$ and
assume that $h_v$ is the hyperplane that defines $C_S$ and contains the edge of $T_S$ of color $i$.
Then for every special point $x \in H_v^- \cap X_c$, we have $\zeta(x) \neq i$.
\item
Let $H_u^-$ and $H_v^-$ be two of the minor halfspaces determining $C_S$.
Then for every special point $x \in H_u^- \cap H_v^- \cap X_c$, we have $\zeta(x) \neq 4$.
\end{enumerate}
\end{lemma}

\begin{proof}

Let $e_u = \overline{vw}$, $e_v = \overline{uw}$ and $e_w = \overline{uv}$ be the edges of $S$.

We first prove part a).
All vertices of $P_n$ have color $4$.
If the minor halfspace $H_v^-$ contains a point of an edge of $P_n$, 
then this edge is incident to $v$.
No point of the edges $e_u$ and $e_w$ has color $i$.
All the points in the intersection of $H_v^-$ and the edges incident to $v$ other than $e_u$ and $e_w$
have color $4$ by (C3).
For every rainbow triangular face $S'$ other than $S$, 
the simplex $\Delta_{S'}$ is in $H_v^+$ by Lemma~\ref{lem:disj-removals}.
All points of color $i$ on the boundary of $\Delta_S$ lie on $h_v$.
Let $S'$ be a rainbow triangular face, including the case $S'=S$.
All points of $S' \cap X_c$ of color $i$ lie in the convex hull of points of color $i$ on one of the edges of $S'$
and one of the edges of $T_{S'}$.
All points on these edges of color $i$ are outside $H_v^-$
and so all the points in $S' \cap X_c$ of color $i$ are outside $H_v^-$.

We now prove part b).
The only vertex of $P_n$ in the halfspace $H_u^-$ is $u$ and the only vertex in $H_v^-$ is $v$
and so the intersection $H_u^- \cap H_v^-$ contains no vertex of $P_n$.
The only edge of $P_n$ with non-empty intersection with $H_u^- \cap H_v^-$ is the edge $\overline{uv}$
and no point of $H_u^- \cap H_v^- \cap \overline{uv}$ has color $4$ by (C3).

Refer to Figure~\ref{fig:R6-closed-proof}.
For every rainbow triangular face $S'$, 
the only points of color $4$ on the boundary of $\Delta_{S'}$ are the three vertices of $T_{S'}$,
Thus all points of $\Delta_{S'}$ of color $4$ lie in $S'$.
If $S'$ has a non-empty intersection with $H_u^- \cap H_v^-$, then $u$ and $v$ are vertices of $S'$.
Observe that no point of $S$ of color $4$ lies in $H_v^- \cap H_u^-$.
If $S' \neq S$, the triangle $T_{S'}$ is neither in $H_v^-$ nor in $H_u^-$ by Lemma~\ref{lem:disj-removals} 
and thus $H_v^- \cap H_u^- \cap S'$ contains only points of the color of the edge $\{u,v\}$.
\end{proof}

\myfig{R6-closed-proof}{1}{The intersection of the minor halfspaces $H_u^-$ and $H_v^-$ on the rainbow triangular face $S$ and a rainbow triangular face $S'$ sharing the edge $\overline{uv}$ with $S$.}

By Lemma~\ref{lem:dim6specpoints} and Observation~\ref{obs:halfspaces}, 
if $r$ and $s$ are special points of the same color, then $\overline{rs} \subseteq X_c$
and so the points $r$ and $s$ are not connected by an edge in the invisibility graph $I(X_c)$.

We now show that $\zeta$ is a proper coloring of $I(X)$.
Let the \emph{rank} of a special point $r \in X_c$ be $0$.
The \emph{rank} of an ordinary point $r \in X_c$ is the dimension of the face $F$ of $P_n$ 
such that $r$ is in the interior of $F \cap X_c$.
For contradiction, suppose that there are points $r, s \in X_c$ of the same color such that
$\overline{rs} \not \subseteq X_c$.
Suppose that among all such pairs, the sum of the ranks of $r$ and $s$ is minimal and that the rank of $r$ is at least as large as the rank of $s$.
Then the rank of $r$ is at least $1$.
Let $F$ be the ordinary face such that $r$ is in the interior of $F \cap X_c$.
Let $r'$ be the point of intersection of the boundary of $F \cap X_c$ 
with the ray emanating from $q_F$ and passing through $r$.
Thus the rank of $r'$ is smaller than the rank of $r$.
Since $F \cap X_c$ was colored by the $q_F$-extension of the coloring of its boundary,
we have $\zeta(r') = \zeta(r)$.
By Observation~\ref{obs:extension} applied with $p = q_F$ and $R = X_c$,
we have that $\overline{r's} \not \subseteq X_c$.
This is a contradiction, 
because the sum of ranks of $r'$ and $s$ is smaller than the sum of ranks of $r$ and $s$.


\section{Constructions in dimension 5}
\label{sec:dim5}
Here we prove Theorem~\ref{thm:dim5}. 
The constructions are similar to those in dimension $6$: 
for part (1) of the theorem, the set $X$ is a closed cyclic polytope 
with one-point holes in some of the $2$-dimensional faces. 
For part (2), instead of points, we remove small $5$-dimensional simplices attached to the 2-dimensional faces. 
The difference from the construction in dimension $6$ is in the placement of the holes: 
here we cannot apply the same argument as in the previous section since for the cyclic polytope 
in dimension $5$ only quadratically many triples of vertices induce a $2$-dimensional face 
and there is a $2$-coloring of the vertex set in which no $2$-dimensional face is monochromatic.

We show two different ways how to choose the holes. 
In the first construction we essentially show that randomly chosen holes will do, 
but the proof (interestingly) requires a rather nontrivial result from group theory
and needs the axiom of choice. 
Also the construction proves only part (1) of the theorem. 
In the second construction we specify the locations of the holes precisely. 
Moreover, we show that the holes can be enlarged to open simplices, which shows part (2) of the theorem.

Let $P_n$ be a $5$-dimensional cyclic polytope on $n\ge 6$ vertices with (ordered) vertex set $V_n=\{v_1,\allowbreak v_2,\allowbreak \dots,\allowbreak v_n\}$. 
Every segment $\overline{v_i v_j}$ with $1 \le i \le j \le n$ forms an edge of $P_n$.
The edges $\overline{v_1 v_i}$ are \emph{edges of type $1i$} and the
edges $\overline{v_i v_j}$ with $2 \le i \le j \le n$ are \emph{edges of type $ij$}.

The $2$-dimensional faces of $P_n$ are the triangles 
\begin{itemize}
 \item 
$v_1 v_i v_j$ for every $1<i<j\le n$ (\emph{triangles of type $1ij$}),
\item
$v_i v_j v_n$ for every $1\le i<j <n$ (\emph{triangles of type $ijn$}),
\item
$v_i v_{i+1} v_j$ for every $1 < i < j-1 < n$ (\emph{triangles of type $i(i+1)j$}) and
\item
$v_i v_j v_{j+1}$ for every $1 < i < j < n-1$ (\emph{triangles of type $ij(j+1)$}).
\end{itemize}

We emphasize that the symbols $i,j$ in the names of the types do not represent any particular numbers.

\subsection{Covering with convex sets}
In the constructions proving part (1) of Theorem~\ref{thm:dim5}, 
we remove a one-point hole from every triangle of type $1ij$.
In the construction proving part (2), 
we remove an open flat simplex instead of the point (as in Section~\ref{sec:dim6}).
The following lemma shows that in both cases, 
the convexity number of the resulting set can be arbitrarily large.

\begin{lemma}
 \label{lem:dim5covering}
Let $X$ be a subset of $P_n$ such that every edge of $P_n$ is a subset of $X$
and none of the triangles of type $1ij$ is a subset of $X$. 
Then $\gamma(X)\ge\Omega(\log{n}/\log\log{n})$.
\end{lemma}

\begin{proof}
Let $X=C_1\cup C_2\cup\dots\cup C_k$ be a covering of $X$ with convex subsets of $X$. The covering induces a partition of each open edge $\overline{v_1 v_i}$, $2\le i\le n$, into $k_i \le k$ intervals $I_i^1,I_i^2,\dots,I_i^{k_i}$, where for each $j$, the interval $I_i^j$ is covered by a convex set $C_{l(i,j)}$ where $1 \le l(i,j) \le k_i$. Since the convex sets in the covering may overlap, this partition need not be unique; in such a case we pick an arbitrary one.

We say that the partitions of two edges $\overline{v_1 v_i}$ and $\overline{v_1 v_{i'}}$ are of the same {\em type\/} if $k_i=k_{i'}$, $l(i,p)=l(i',p)$ for each $p=1,2,\dots, k_i$ (in other words, the ``colors'' appear in the same order along the edges), and for each $p=1,2,\dots, k_i$ the {\em type\/} of the interval $I_i^p$ (that is, closed, open, or half-closed from the left/right) is the same as the type of the interval $I_{i'}^p$. Degenerate one-point intervals are considered as closed. The number of types of the partitions is at most $2^k\cdot k!\cdot 2^{k-1}$. Indeed, there are at most $2^k$ subsets of $\{C_1, C_2, \dots, C_k\}$, each of the subsets can be linearly ordered in at most $k!$ ways, and there are at most $k-1$ boundary points shared by two intervals, where one of the intervals is locally closed and the other one locally open.

It follows that if $n>2^k\cdot k!\cdot 2^{k-1}+1$, then there are two edges $\overline{v_1 v_i}$ and $\overline{v_1 v_{i'}}$ with partitions of the same type. The convex hulls $\mathrm{conv}(I_i^p\cup I_{i'}^p)$ cover the whole open triangle $v_1 v_i v_{i'}$, including the one-point hole inside, which is a contradiction. Therefore $n \le 2^k\cdot k!\cdot 2^{k-1}+1$, which implies that $\gamma(X)\ge\Omega(\log{n}/\log\log{n})$.
\end{proof}

\subsection{The first construction}
 The set $X$ is obtained from $P_n$ by making a one-point hole in the interior of each triangle of type $1ij$, in such a way that the $2$-dimensional coordinates of the holes, relative to the generating vectors $v_i - v_1$ and $v_j - v_1$, are algebraically independent.

It remains to show that three colors suffice to properly color the invisibility graph $I(X)$. Observe that the interiors of faces of dimensions at least $3$ consist entirely of isolated vertices in $I(X)$ and thus can be colored with one color, independently of the rest of the graph. Also observe that the interiors of $2$-dimensional faces form a bipartite subgraph of $I(X)$. The main difficulty lies in coloring the edges of $P_n$, since they may induce odd cycles in the invisibility graph.

Let $e_i$ denote the half-open segment $\overline{v_1 v_i} \setminus \{v_1\}$ and let $h_{i,j}$ denote the one-point hole in the face $v_1 v_i v_j$.
Let $H$ be the subgraph of $I(X)$ induced by the union of the segments $e_i$, $2\le i \le n$.

\begin{observation}\label{obs_stupne}
Each vertex $w$ of $H$ has degree at most $n-1$. In particular, if $w\in e_i$, then $w$ has at most one neighbor on every edge $e_j$ with $j\neq i$.
\end{observation}

\begin{proof}
For each neighbor $u$ of $w$, the segment $\overline{wu}$ passes through a one-point hole $h_{i,j}$, for some $j\neq i$. There are $n-1$ such one-point holes. The observation then follows from the fact that the ray $wh_{i,j}$ intersects $V(H)\setminus \{w\}$ in at most one point, which lies on the edge $e_j$.
\end{proof}

Observation~\ref{obs_stupne} implies that each connected component of $H$ is countable, therefore $H$ has a continuum connected components.
In the next observation we show, in particular, that $H$ has only countably many odd cycles, which implies that almost all components of $H$ are bipartite. Then we show that each component with an odd cycle is $3$-colorable.

In the rest of the section we will identify each edge $e_i$ with the half open interval $(0,1]$ by an affine map that sends the vertex $v_1$ to $0$ and the vertex $v_i$ to $1$.

For $i,j =2,3,\dots, n$, $i\neq j$, let $N_{i,j} : e_i \cup \{\lambda\} \rightarrow e_j \cup \{\lambda\}$ be the function that assigns to each point $w \in e_i$ its neighbor in $e_j$. In case such a neighbor does not exist (the ray from $w$ through the hole $h_{i,j}$ does not intersect $e_j$), we let $N_{i,j}(w)\mathrel{\mathop:}=\lambda$. We also define $N_{i,j}(\lambda)\mathrel{\mathop:}=\lambda$.

\begin{observation}
\begin{enumerate}
\item[{\rm 1)}] Each function $N_{i,j}$ is strictly decreasing on $N_{i,j}^{-1}[e_j]$.
\item[{\rm 2)}] A composition $N_{i_1,i_2,\dots, i_{k+1}}$ of an odd number of functions $N_{i_1,i_2}$, $N_{i_2,i_3}$, \dots, $N_{i_k,i_{k+1}}$ is strictly decreasing on $N_{i_1,i_2,\dots, i_{k+1}}^{-1}[e_{k+1}]$. Therefore if $i_1=i_{k+1}$, then $N_{i_1,i_2,\dots, i_{k+1}}$ has at most one fixed point in $e_{i_1}$.
\item[{\rm 3)}] For each sequence $i_1,i_2,\dots,i_k$ of odd length there is at most one cycle $w_1w_2 \dots w_k$ in $H$ with $w_j\in e_{i_j}$ for each $j=1,2,\dots,k$.
\end{enumerate}
\qed
\end{observation}

Now we find a precise form of the functions $N_{i,j}$, which will allow us to determine all the odd cycles in $H$.

The edges $e_i$ and $e_j$ determine a canonical coordinate system in the triangle $v_1 v_i v_j$, by an affine map to $\mathbb{R}^2$ that sends the vertex $v_1$ to $(0,0)$, $v_i$ to $(1,0)$ and $v_j$ to $(0,1)$. Let $(x_{i,j},y_{i,j})$ be the coordinates of the hole $h_{i,j}$.

\begin{observation}\label{obs_Nij}
Suppose that $N_{i,j}(x)=y$ for some $x,y\in (0,1]$. Then
\[
y=y_{i,j} + \frac{x_{i,j}y_{i,j}}{x-x_{i,j}} = \frac{y_{i,j}x}{x-x_{i,j}}.
\]
\end{observation}

\begin{proof}
Refer to Figure~\ref{obr_souradnice}. By the similarity of the two shaded triangles, we have 
\[
\frac{x-x_{i,j}}{y_{i,j}}=\frac{x_{i,j}}{y-y_{i,j}}
\]
and the formula follows.
\end{proof}

\begin{figure}
\begin{center}
\ifpdf\epsfbox{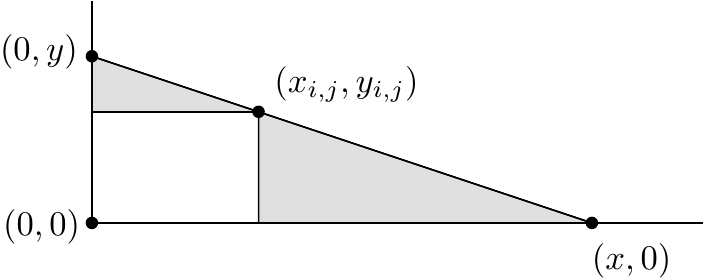}\fi
\end{center}
\caption{Deriving the formula for the function $N_{i,j}$.}
\label{obr_souradnice}
\end{figure}

Observation~\ref{obs_Nij} shows that $N_{i,j}$ is a fractional linear transformation (when restricted to the segment $N_{i,j}^{-1}[e_j]$). Therefore, we can alternatively describe $N_{i,j}$ as a projective map of the real projective line, or by a $2\times 2$ matrix
\[
M_{i,j}=\begin{pmatrix} y_{i,j} & 0 \\ 1 & -x_{i,j} \end{pmatrix},
\]
which satisfies
\[
M_{i,j}\begin{pmatrix} x\\1 \end{pmatrix} = \begin{pmatrix} yz\\z \end{pmatrix}
\]
for some nonzero $z$, which depends on $x$.
Since scalar multiples of $M_{i,j}$ determine the same projective map, we may choose a representing matrix $M'_{i,j}$ with determinant $-1$:

\[
M'_{i,j}\mathrel{\mathop:}=M_{i,j}\cdot\sqrt{x_{i,j}y_{i,j}} = \begin{pmatrix} c_{i,j} & 0 \\ d_{i,j} & -c^{-1}_{i,j} \end{pmatrix}
\]
where $c_{i,j}=\sqrt{x_{i,j}/y_{i,j}}$ and $d_{i,j}=\sqrt{x_{i,j}y_{i,j}}$. Note that $M'_{j,i}$ is the inverse of $M'_{i,j}$.
Composition of maps $N_{i,j}$ now corresponds to multiplication of the matrices $M'_{i,j}$.

\begin{observation}\label{obs_vlastnivektory}
A map $N_{i_1,i_2,\dots, i_{k},i_1}$ has a fixed point $x$ if and only if
$\left(\begin{smallmatrix} x \\ 1 \end{smallmatrix}\right)$
is an eigenvector of the matrix $M'_{i_1,i_2,\dots, i_{k},i_1}\mathrel{\mathop:}=M'_{i_k,i_1}\cdots M'_{i_2,i_3}M'_{i_1,i_2}$.
\qed
\end{observation}

\begin{lemma}\label{lem_ABBA}
Let $A$ and $B$ be two lower-triangular $2\times 2$ matrices that share an eigenvector $v_x=\left(\begin{smallmatrix} x \\ 1 \end{smallmatrix}\right)$, for some $x\neq 0$. Then $AB=BA$.
\end{lemma}

\begin{proof}
If $A$ is diagonal and has an eigenvector $v_x$ with $x\neq 0$, then $A$ is a scalar multiple of the identity matrix, and so it commutes with any $2\times 2$ matrix. The case of $B$ being diagonal is symmetric. From now assume that neither of $A$ or $B$ is diagonal.

Let $A = \left(\begin{smallmatrix} a & 0 \\ b & c \end{smallmatrix}\right)$ and  $B = \left(\begin{smallmatrix} a' & 0 \\ b' & c' \end{smallmatrix}\right)$.
Then $AB=BA$ if and only if $ba'+cb'=b'a+c'b$, equivalently, $b(a'-c')=b'(a-c)$.
Since $x\neq 0$, the vector $v_x$ is an eigenvector of $A$ if and only if $a=bx+c$, equivalently, $x=(a-c)/b$ (note that $b\neq 0$ as $A$ is not diagonal by our assumption). Similarly, $v_x$ is an eigenvector of $B$ if and only if $x=(a'-c')/b'$. Therefore, $A$ and $B$ share an eigenvector $v_x$, for some nonzero $x$, if and only if $(a-c)/b=(a'-c')/b'$, equivalently, $b(a'-c')=b'(a-c)$.
\end{proof}

We will use a few algebraic results about free metabelian groups. The {\em free metabelian group with $m$ generators\/} is the quotient of the free group $F_m$ with generators $x_1,x_2,\dots,x_m$ by the second derived subgroup $F''_m$. For every $i\in\{1,2,\dots,m\}$ and every element $a\in F_{m}$ written as $a=x_{i_1}^{k_1}x_{i_2}^{k_2}\dots x_{i_n}^{k_n}$, where
$n\in\mathbb{N}$, $i_j\in\{1,2,\dots,m\}$ and $k_j\in\mathbb{Z}$, we define the {\em total power of $x_i$ in $a$} as the number $\sum_{j\in\{1,2,\dots,n\};i_j=i}k_j$, which is independent on the chosen expression of $a$. The commutator subgroup $F'_{m}$ then consists precisely of those elements $a$ of $F_m$ such that the total power of every generator $x_i$ in $a$ is zero. 
For this reason we can analogously define the total power of generators also for the elements of the free metabelian group. 

The following representation of free metabelian groups by $2\times 2$ matrices was found by Magnus~\cite{Magnus} and is usually called the {\em Magnus embedding}.

\begin{theorem}
\label{magnus}
{\rm \cite{Magnus} (see also{~\cite{DruKap13_lectures, Du69_magnus, Gupta}})}
Let $R=\mathbb{Z}[x_{1},x_{2},\dots,x_{m},y_{1},y_{2},\dots,y_{m}]$ be the ring of polynomials over $\mathbb{Z}$ with variables $x_{1},x_{2},\dots,x_{m},y_{1},y_{2},\dots,y_{m}$. Let $M_{2}(R)$ be the ring of all $2\times 2$ matrices with entries from $R$. Let $H$ be a free metabelian group with basis $X_{1},X_{2},\dots,X_{m}$. Then
there is a faithful representation $\varphi : H \to M_{2}(R)$ such that
$\varphi(X_{i}) =(\begin{smallmatrix}x_{i}& y_{i}\\0&1\end{smallmatrix})$ for every $i=1,2,\dots,m$.
\end{theorem}

\begin{theorem}\label{metabelian}
Suppose that the entries $c_{i,j},d_{i,j}$ in the matrices $M'_{i,j}$, $2\le i < j\le n$, are algebraically independent. Then the matrices $M'_{i,j}$, $2\le i < j\le n$, generate a free metabelian group $G$.
\end{theorem}

\begin{proof}
Let $G$ be the group generated by the matrices $M'_{i,j}$, $2\le i < j\le n$.
Considering an automorphism $\psi$ of $\mathrm{GL}_{2}(\mathbb{R})$ given by $\psi(A)=(A^{T})^{-1}$, we get that $G$ is isomorphic to $K=\psi(G)$, which is generated by the matrices $P_{ij}=\psi(M'_{ij})=\Big(\begin{smallmatrix}e_{i,j}& f_{i,j}\\0&-e_{i,j}^{-1}\end{smallmatrix}\Big)$, $2\le i < j\le n$, where $e_{i,j}=c_{i,j}^{-1}$ and $f_{i,j}=d_{i,j}$, $2\le i < j\le n$, are again algebraically independent. Let $\mathrm{UT}_{2}(\mathbb{R})$ be the subgroup of $\mathrm{GL}_{2}(\mathbb{R})$ consisting of all upper triangular matrices. Since the map $\pi:\mathrm{UT}_{2}(\mathbb{R})\to \mathrm{UT}_{2}(\mathbb{R})$ defined as $\pi\big((\begin{smallmatrix}a& b\\0&c\end{smallmatrix})\big)=\big(\begin{smallmatrix}a/c& b/c\\0&1\end{smallmatrix}\big)$ is a homomorphism, $K\subseteq\mathrm{UT}_{2}(\mathbb{R})$ and $\ker(\pi)\cap K=1$, we obtain that $K$ is isomorphic to $\pi(K)$, which is generated by the matrices $\pi(P_{ij})=\Big(\begin{smallmatrix}x_{i,j}& y_{i,j}\\0&1\end{smallmatrix}\Big)$, $2\le i < j\le n$, where $x_{i,j}=-e_{i,j}^{2}$ and $y_{i,j}=-e_{i,j}f_{i,j}$. 
Let $\mu:\mathbb{Z}[\{X_{i,j},Y_{i,j}|2\le i < j\le n\}]\to \mathbb{Z}[\{E_{i,j},F_{i,j}|2\le i < j\le n\}]$ be a substitution defined as $\mu(X_{i,j})=-E_{i,j}^{2}$ and $\mu(Y_{i,j})=-E_{i,j}F_{i,j}$. Since $\mu$ is injective on the set of all monomials (with coefficients), it is a monomorphism.
Hence $x_{i,j}$ and $y_{i,j}$, $2\le i < j\le n$, are again algebraically independent.

Theorem~\ref{magnus} now implies that $\pi(K)$ is a free metabelian group. 
\end{proof}

The following Theorem was proved by Malcev~\cite{Malcev}.

\begin{theorem}\label{malcev}{\rm \cite{Malcev} (see also~\cite[Chapter II, 1.14]{Gupta})}
Let $G$ be a free metabelian group and $G'$ its commutator subgroup. Two elements $a,b \in G\setminus G'$ commute if and only if they are contained in a common cyclic subgroup of $G$. That is, there exists $c\in G\setminus G'$ such that $a=c^m$ and $b=c^n$ for some pair of integers $m,n$.
\end{theorem}

An {\em $(e_i,e_j)$-edge\/} is an edge $uv$ of $H$ with $u\in e_i$ and $v \in e_j$. Note that as the edges of $H$ are not oriented, each $(e_i,e_j)$-edge is also an $(e_j,e_i)$-edge.

\begin{lemma}\label{lem_lichecykly}
Let $x$ be a vertex of $H$ and let $C[x]$ be the component of $H$ containing $x$. There exist $i,j$ such that every odd cycle in $C[x]$ contains an $(e_i,e_j)$-edge.
\end{lemma}

\begin{proof}
Every odd cycle $C$ in $C[x]$ can be extended to an odd closed walk $PCP^{-1}$ starting and ending at $x$, where $P$ is a path from $x$ to a vertex of $C$. We will show that there exist $i,j$ such that every odd closed walk passing through $x$ has an odd number of $(e_i,e_j)$-edges, from which the lemma follows.

Let $W=(w_1=x,w_2, \dots,w_k,x)$ and $W'=(w'_1=x,w'_2, \dots,w'_{k'},x)$ be two odd closed walks starting at $x$. For each $j=1,2,\dots,k$, let $e_{i_j}$ be the segment containing $w_j$. Similarly, for each $j=1,2,\dots,k'$, let $e_{i'_j}$ be the segment containing $w'_j$. Then by Observation~\ref{obs_vlastnivektory}, the matrices $M_W=M'_{i_1,i_2,\dots, i_{k},i_1}$ and $M_{W'}=M'_{i'_1,i'_2,\dots, i'_{k'},i'_1}$ share the eigenvector $\left(\begin{smallmatrix} x \\ 1 \end{smallmatrix}\right)$. Since $x>0$, Lemma~\ref{lem_ABBA} implies that $M_W$ and $M_{W'}$ commute.

Let $G$ be the group generated by the matrices $M'_{i,j}$. By Theorem~\ref{metabelian}, $G$ is a free metabelian group with generators $M'_{i,j}$, $2 \le i < j \le n$.
Since both matrices $M_W$ and $M_{W'}$ are a product of an odd number of the generators or their inverses, they lie outside of the commutator subgroup $G'$.

By Theorem~\ref{malcev}, there is a matrix $L\in G$ such that $M_{W}=L^m$ and $M_{W'}=L^n$, for some integers $m,n$. 
This implies that $L$ is a product of an odd number of generators, $n$ and $m$ are odd, and for each generator $M'_{i,j}$, the total powers of $M'_{i,j}$ in $L,M_W$ and $M_{W'}$ have the same parity. In particular, there is a matrix $M'_{i,j}$ that has an odd total power in $M_{W}$ and hence it has an odd total power in $M_{W'}$ for every other odd closed walk $W'$ passing through $x$. This means that all odd closed walks $W$ passing through $x$ have an odd number of $(e_i,e_j)$-edges.
\end{proof}

Lemma~\ref{lem_lichecykly} implies that $C[x] \setminus e_i$ is a bipartite graph. Since $C[x]\cap e_i$ is an independent set, the whole component $C[x]$ can be properly colored with three colors.
Consequently, the whole subgraph $H$ is $3$-colorable. We fix one such coloring using colors $1,2$ and $3$.

It remains to define the coloring on the rest of the points of $X$. First we assign color $1$ to  the vertex $v_1$ (we may do this as the vertex $v_1$ can see all points of $H$). Next we color the interiors of the edges of type $ij$ and the interiors of the $2$-dimensional faces of type $1ij$, using colors $1$ and $2$ only. Finally we color the remaining points of $X$, which are isolated in $I(X)$, with color $1$.


\subsection{The second construction}

The \emph{median} of a triangle is a line segment joining a vertex to the midpoint of the opposite side.

\begin{observation}
\label{obs:R6trianglepoint}
Let $T$ be a triangle with vertices $A$, $B$ and $C$. Let $AM$ be a median of $T$ and let $P \neq M$ be a point on $AM$.
Let $P_0$ be the intersection of the line $AB$ 
and the line parallel to $BC$ that passes through $P$.
Let $X$ be the intersection of the line $AB$ 
and the line $CP$.
Then $\|B - P_0\| > \|P_0 - X\|$.
\end{observation}

\begin{proof}
Refer to Figure~\ref{fig:dim5lowpoint}.
Take the parallelogram $B'C'CB$ that has $BC$ as one of its edges 
and $P$ as the intersection of the diagonals.
Then $BB'$ is parallel to $MA$ and so $P_0$ lies in the interior of the parallelogram $B'C'CB$.
Then $X$ also lies in the interior of the parallelogram 
and so its distance from $P_0$ is smaller than the distance between $P_0$ and $B$.
\end{proof}

\myfig{dim5lowpoint}{1}{An illustration the proof of Observation~\ref{obs:R6trianglepoint}.}

The \emph{relative height} of a point $p \in \overline{v_{1} x}$, 
where $x$ is a point on one of the edges of type $ij$, 
is the ratio of the distances of $x$ from $p$ and from $v_1$:
\[
\relheight(p) \mathrel{\mathop:}= \frac{\lVert p-x\rVert}{\lVert v_1-x\rVert }.
\]
For every $i$ and $j$ such that $2 \le i < j \le n$,
we let $p_{ij}$ be the point on the median between $v_1$ and the midpoint $m_{ij}$ of the edge $v_i v_j$
with relative height 
\[
\relheight(p_{ij}) = \frac{3^{ni+j-n^2}}{2}.
\]

We now prove the first part of Theorem~\ref{thm:dim5}. 

\begin{theorem}
\label{thm:dim5SecondNonclosed}
Let $X \mathrel{\mathop:}= P_n \setminus \{p_{ij}, 2\le i<j\le n\}$.
We have
\[
\gamma(X) \ge \Omega(\log n /\log \log n) \qquad \text{and} \qquad \chi(X) = 2.
\]
\end{theorem}
\begin{proof}
The bound $\gamma(X) \ge \Omega(\log n /\log \log n)$ follows from Lemma~\ref{lem:dim5covering}.

For simplicity of notation, we define $p_{k l} \mathrel{\mathop:}= p_{l k}$ for every $k>l$.

Let $G$ be the subgraph of the invisibility graph $I(X)$ 
induced by the points on the edges of type $1i$.
We first show that $G$ is acyclic and so it can be colored by two colors.
For contradiction, assume that $G$ contains a cycle.
Let $s$ be the point of the cycle of maximum relative height
and let $t$ and $t'$ be its neighbors along the cycle.
Let $k$, $l$ and $l'$ be the distinct indices such that
$s \in \overline{v_1v_k}$, $t \in \overline{v_1 v_l}$ and $t' \in \overline{v_1v_{l'}}$.
We have $p_{k l} \in \overline{st}$ and $p_{k l'} \in \overline{st'}$.
By Observation~\ref{obs:R6trianglepoint}, 
$\relheight(p_{k l}), \relheight(p_{k l'}) \in [\relheight(s)/2, \relheight(s)]$.
This is a contradiction since the ratio between the relative heights of any two points 
from $\{p_{ij}, 2 \le i < j \le n\}$ is at least $3$.

We can thus color all the points on the edges of type $1i$ with two colors.
Every point in the interior of an edge of type $ij$ is connected in $I(X)$ 
to exactly one already colored point and is not connected to any point on an edge of type $ij$.
Therefore we can extend the coloring to all the points on all the edges of $X$.

Let $q$ be a point in the interior of a triangular face $S$ with vertices $v_1$, $v_k$ and $v_{l}$.
Then the ray emanating from $p_{k l}$ and passing through $q$ intersects one of the edges of $S$.
We color $q$ by the color of the point in the intersection.
This produces a proper coloring of the subgraph of $I(X)$ induced by the points 
on the triangular faces of type $1ij$.

All points of $X$ lying outside the triangular faces of type $1ij$ are vertices of degree zero in $I(X)$
and so they can be colored by color $1$.
\end{proof}

Note that the presented proof of Theorem~\ref{thm:dim5SecondNonclosed} relies on the axiom of choice.
However, the proof of the second part of Theorem~\ref{thm:dim5} below can be modified to construct,
in a finite number of steps, a $2$-coloring of $I(P_n \setminus \{p_{ij}, 2 \le i < j \le n\})$ 
with two colors, thus providing a proof without the need for the axiom of choice.

Let $I \subseteq [0,1]$ be an interval.
To simplify the notation, we say that a point $x \in \overline{v_1v_i}$ lies in the \emph{$I$-interval} 
of the edge $\overline{v_1v_i}$ if $\relheight(x) \in I$.
The \emph{$(k,l)$-piece} of an edge of type $1i$ is the 
$[3^{kn+l-1-n^2}, 3^{kn+l-n^2})$-interval of the edge.

Assume that a line segment with endpoints $x \in \overline{v_1v_i}$ and $y \in \overline{v_1v_j}$ 
contains $p_{ij}$.
By Observation~\ref{obs:R6trianglepoint}, $x$ and $y$ lie in the $[0,2 \cdot\relheight(p_{ij}))$-intervals 
of the edges $\overline{v_1v_i}$ and $\overline{v_1v_j}$, respectively.

Given $\varepsilon > 0$, we let $T_{ij}(\varepsilon)$ be the image of the triangular face $v_1v_iv_j$ 
under the homothety with center $p_{ij}$ and ratio $\varepsilon$.
Since $2\cdot\relheight(p_{ij}) = 3^{ni+j-n^2} $, we have the following corollary of Observation~\ref{obs:R6trianglepoint}.

\begin{corollary}
\label{cor:R6ratios}
For every $i$ and $j$ satisfying $2 \le i < j \le n$, there is $\varepsilon'_{ij} > 0$ and 
$\bar{\delta}_{i,j} > 0$ satisfying the following.
Whenever a line segment with endpoints $x \in \overline{v_1v_i}$ and $y \in \overline{v_1v_j}$ 
intersects $T_{ij}(\varepsilon'_{ij})$,
then $x$ and $y$ lie in the $[0,3^{ni+j-n^2})$-intervals of $\overline{v_1v_i}$ and $\overline{v_1v_j}$,
respectively.
Additionally, at most one of $x$ and $y$ lies in the 
$[0,3^{ni+j-1-n^2}+\bar{\delta}_{i,j}]$-interval 
of the edge of type $1i$ that contains it.
\qed
\end{corollary}

The construction of the closed set $X_c$ proving the second part of Theorem~\ref{thm:dim5}
is similar to the construction of the closed set proving the second part of Theorem~\ref{thm:dim6}.
For every triangular face of type $1ij$, we take the triangle $T_{ij}(\varepsilon_{ij})$ 
for an appropriately chosen $\varepsilon_{ij}$.
Then we remove from $P_n$, for every triangular face of type $1ij$, 
a flat open simplex $\Delta_{ij}$ having $T_{ij}(\varepsilon_{ij})$ as a triangular face.
However, the requirements posed on the flatness of the simplices are different 
from the proof of the second part of Theorem~\ref{thm:dim6}.

We start by defining the values $\varepsilon_{ij}$
and a coloring of the points on the edges of $P_n$ of type $1i$.
The \emph{$\delta$-neighborhood} of a point $x \in \mathbb{R}^5$ is the set of points 
of $\mathbb{R}^5$ at distance less than $\delta$ from $x$.

\begin{lemma}
\label{lem:dim5SecondClosedPart1}
There are $\delta > 0$, $\varepsilon_{ij} \in (0,1)$ for every $i,j$ satisfying $2 \le i < j \le n$, 
and a coloring of the points on the edges of $P_n$ of type $1i$ using 
colors from $\{1,2,3\}$ satisfying the following.
No segment between two points in $\delta$-neighborhoods of equally colored points
intersects any of the triangles $T_{ij}(\varepsilon_{ij})$.
\end{lemma}

\begin{proof}
We first color all the points in the $[0, 3^{2n-n^2})$-interval of every edge of type $1i$ by color $1$.
We then proceed in steps $(k,l)$, $k \in \{2, 3, \dots, n\}$, $l \in \{1, 2, \dots, n\}$ 
in the lexicographic order.
In the step $(k,l)$, we color the $(k,l)$-pieces of all the edges of type $1i$
and fix $\varepsilon_{kl}$.
During the process, we guarantee that after each step, 
the points in each color class form a union of finitely many 
nondegenerate subsegments of the edges of type $1i$.

We say that a segment $\overline{pq}$ is \emph{blocked} if $\overline{pq}$ intersects 
some of the triangles $T_{ij}(\varepsilon'_{ij})$ where $2 \le i < j \le n$ and $\varepsilon'_{ij}$
is from Corollary~\ref{cor:R6ratios}.

When considering a single step $(k,l)$, 
the points colored before this step are the points in the 
$[0,3^{kn+l-1-n^2})$-intervals of the edges of type $1i$.
These are the \emph{old points}.
The \emph{new points} are the points in the $(k,l)$-pieces of the edges of type $1i$.

When $l \le k$, then in the step $(k, l)$, we color all the new points by color $1$.
Let $a$ be a new point.
Let $b$ be a new point or an old point.
By Corollary \ref{cor:R6ratios}, when we use $\delta_{k,l} = \bar{\delta}_{k,l}$, no blocked segment 
has endpoints in the $\delta_{k,l}$-neighborhoods of $a$ and $b$.

We now describe a step $(k, l)$ where $l > k$.

By Corollary~\ref{cor:R6ratios}, if a segment between points $q$ and $r$ is blocked and $q$ is new
and $r$ is either old or new, then one of $q$ and $r$ lies on $\overline{v_1 v_{k}}$
and the other on $\overline{v_1 v_{l}}$.
We color all the new points on the edges $\overline{v_1 v_{i}}$, 
where $i \in \{2,3,\dots,n\} \setminus \{k,l\}$, by color $1$.

\myfig{dim5colEdges}{1}{A coloring of the points in the $(k,l)$-piece of the edge $\overline{v_1 v_k}$. The points in the segments $A, B$ and $C$ are not allowed to be colored by color $1$, $2$ and $3$, respectively.}

We take a coloring of the $(k,l)$-piece of $\overline{v_1 v_{k}}$ with colors $\{1,2,3\}$
such that some $\delta'_{k,l} > 0$ and $\varepsilon_{k,l}''$ with 
$\varepsilon_{k,l}' > \varepsilon_{k,l}'' > 0$ satisfy the following.
For every color $c$, every segment between 
a point in the $\delta'_{k,l}$-neighborhood of the newly colored points of color $c$ 
and a point in the $\delta'_{k,l}$-neighborhood of old points of color $c$ 
is disjoint with $T_{k,l}(\varepsilon''_{k,l})$.
In addition, each color class is composed of finitely many nondegenerate segments.
See Figure~\ref{fig:dim5colEdges}.

We continue by coloring the new points on the edge $\overline{v_1 v_{l}}$.
We take a coloring of the new points on $\overline{v_1 v_{l}}$ with colors $\{1,2,3\}$
such that some $\delta_{k,l} \in (0, \delta'_{k,l}]$ and 
$\varepsilon_{k,l} \in (0, \varepsilon_{k,l}'')$ satisfy the following.
For every color $c$, every segment between a point in the $\delta_{k, l}$-neighborhood 
of a new point of color $c$ on $\overline{v_1 v_{l}}$ and a point in 
the $\delta_{k,l}$-neighborhood of an old or new point of color $c$ on 
$\overline{v_1 v_{k}}$ is disjoint with $T_{k,l}(\varepsilon_{k,l})$.
In addition, each color class is composed of finitely many nondegenerate segments.

Finally we take $\delta \mathrel{\mathop:}= \min_{k \in \{2,3,\dots, n\},l \in \{1,2,\dots, n\}}\delta_{k, l}$.
\end{proof}

\begin{theorem}
\label{thm:dim5SecondClosed}
For every $2 \le i < j \le n$, there are $\varepsilon_{ij} \in (0,1)$ and open simplices $\Delta_{ij}$ 
having $T_{ij}(\varepsilon_{ij})$ as a triangular face, disjoint with the edges of $P_n$, 
and satisfying the following.
Let $X_c \mathrel{\mathop:}= P_n \setminus \bigcup_{2 \le i < j \le n} \Delta_{ij}$.
Then we have
\[
\gamma(X_c) \ge \Omega(\log n /\log \log n) \qquad \text{and} \qquad \chi(X_c) \le 6.
\]
\end{theorem}

\begin{proof}
The bound on $\gamma(X_c)$ follows from Lemma~\ref{lem:dim5covering}.

We fix the coloring of the edges of type $1i$ and the values $\varepsilon_{ij}$ 
obtained in Lemma~\ref{lem:dim5SecondClosedPart1}. 

All points on the edges of type $ij$ are colored by color $4$.
All other points in triangular faces of type $1ij$ at distance at most $\delta$ from an edge of type $1i$ 
are colored by the color of a nearest point on an edge of type $1i$.
By the coloring and the choice of $\delta_{k, l}$, 
every segment with endpoints in two colored points of the same color from $\{1,2,3\}$ 
is disjoint with all the triangles $T_{k,l}(\varepsilon_{k,l})$.

The triangles $T_{ij}(\varepsilon_{ij})$ are fixed as well, 
and so are the intersections $S \cap X_c$ for every triangular face $S$ of type $1ij$.
We continue by coloring all the uncolored points in these intersections.

Refer to Figure~\ref{fig:dim5colTriangle}.
Let $S$ be a triangular face of type $1ij$ with vertices $v_1, v_k$ and $v_l$
and let $T_S \mathrel{\mathop:}= T_{k,l}(\varepsilon_{k,l})$.
To color the remaining points of $S \setminus T_S=S \cap X_c$,
we split the face by three segments connecting the vertices of $S$ 
with the corresponding vertices of $T_S$ into three trapezoids.
All the uncolored points in the trapezoid incident with $v_k$ and $v_l$ get color $4$, 
all the uncolored points in one of the other two trapezoids get color $5$,
and all the uncolored points in the third trapezoid get color $6$.

\myfig{dim5colTriangle}{1}{A coloring of the triangular faces of type $1ij$.}

A face $F$ of $P_n$, including the case $F=P_n$, is \emph{ordinary} 
if its dimension is at least $2$ and it is not a triangular face of type $1ij$.
For every ordinary face $F$, we fix a point $q_F$ in its interior.

Let $S$ be a triangular face of $P_n$ of type $1ij$ and let $u$, $v$ and $w$ be its vertices.
For every $x \in \{u,v,w\}$,
we let $e_x$ be the edge of $S$ that does not contain the vertex $x$ 
and $e'_x$ the edge of $T_S$ corresponding to $e_x$.
We take three hyperplanes $h_u$, $h_v$ and $h_w$ 
satisfying the following four conditions for every $x \in \{u,v,w\}$.

\begin{enumerate}
\item[(C1)]
The hyperplane $h_{x}$ contains the edge $e'_x$ of $T_S$.
\item[(C2)]
One open halfspace determined by $h_{x}$ contains $x$ and the interior of $T_S$, 
this halfspace is called the \emph{minor halfspace} $H_x^-$. 
The other open halfspace is the \emph{major halfspace} $H_x^+$ and contains 
the point $q_F$ for every ordinary face $F$,
and all the vertices of $P_n$ except for $x$.
\item[(C3)]
For every point $r$ of a triangular face $S' \neq S$ of type $1ij$,
if $r \in H_x^-$, then $r$ is at distance at most $\delta$ from one of the two edges of $S$ of type $1i$.
\item[(C4)]
The orthogonal projection of $P_n \cap H_u^- \cap H_v^- \cap H_w^-$ on the $2$-dimensional affine 
space containing $S$ is a subset of $T_S$.
\end{enumerate}

The construction continues in the same way 
as the construction of the closed set in dimension $6$ in Section~\ref{subsec:dim6closed}.
Let $\mathcal{S}$ be the set of all triangular faces of type $1ij$.
Given $S \in \mathcal{S}$ with vertices $u$, $v$ and $w$, let $C_S \mathrel{\mathop:}= H_u^- \cap H_v^- \cap H_w^-$
and let $X_c \mathrel{\mathop:}= P_n \setminus \bigcup_{S \in \mathcal{S}}{C_S}$.
For every $S \in \mathcal{S}$, let $\Delta_S \mathrel{\mathop:}= P_n \cap C_S$.
For every $x \in \{u,v,w\}$, the hyperplane $h_{x}$ is assigned the color of the points 
on the edge $e'_x$.
Every point on the common boundary of $\Delta_S$ and $X_c$ lies on one, two or three of the hyperplanes $h_{u}$, $h_{v}$ and $h_{w}$, and gets the color of one of these hyperplanes on which it lies.

All the vertices and edges of $P_n$ are contained in some triangular face of type $1ij$
and so it remains to color the points in the interiors of the ordinary faces.
We take all the ordinary faces in the order of nondecreasing dimension.
For every ordinary face $F$, we color the interior of $F \cap X_c$ by the $q_F$-extension 
of the coloring of the boundary of $F \cap X_c$.

A \emph{special point} is a point of $X_c$ that is not in the interior of $F \cap X_c$ 
for any ordinary face $F$ of $P_n$.
First, we show that for every pair $s, t$ of special points of the same color $c \in \{1,2,3\}$,
the segment $\overline{st}$ is contained in the set $X_c$.
Since the color of $s$ and $t$ is from $\{1,2,3\}$, each of them is at distance at most $\delta$ 
from some edge of type $1i$.
Let $S$ be a triangular face of type $1ij$.

If $s$ and $t$ are at distance at most $\delta$ from the edges of $S$ of type $1i$, 
then we let $s'$ and $t'$ be their orthogonal projections on the $2$-dimensional affine space containing $S$.
Then $s'$ and $t'$ are at distance at most $\delta$ from the edges of $S$ of type $1i$,
and so the segment $\overline{s't'}$ is disjoint from $T_S$ by Lemma~\ref{lem:dim5SecondClosedPart1}.
By~(C4), $\overline{st}$ is disjoint with $\Delta_S$.

Otherwise, assume that $s$ is at distance larger than $\delta$ from the two $1i$-edges of $S$.
Then $s$ is in all the three major halfspaces associated with $S$ and so, 
by Observation~\ref{obs:halfspaces}, $\overline{st}$ is disjoint with $\Delta_S$.

Let $s$ and $t$ be two special points of color $c \in \{4,5,6\}$.
Let $S$ be a triangular face of type $1ij$.
If both $s$ and $t$ are in $S \cup \Delta_S$, then both $s$ and $t$ are in $h_x \cup H_x^+$
where $x$ is the vertex such that $e'_x$ has color $c$.
Thus $\overline{st}$ is disjoint with $\Delta_S$.
Otherwise, assume that $s$ does not lie in $S \cup \Delta_S$.
By (C3), $s$ does not lie in any minor halfspace associated with $S$ and so
$\overline{st}$ is disjoint with $\Delta_S$.

Thus, if $s$ and $t$ are special points of the same color, then $\overline{st} \subseteq X_c$.
We apply Observation~\ref{obs:extension} in the same way as in the proof 
of the second part of Theorem~\ref{thm:dim6} in Section~\ref{sec:dim6}
to conclude that if some two points $s, t \in X_c$ have the same color, 
then $\overline{st} \subseteq X_c$.
\end{proof}


\section{Concluding remarks}

To solve Problem~\ref{prob:sep34} in dimension $4$, we could use a construction similar to those
in dimensions $5$ and $6$ provided the following problem has a positive answer.

\begin{problem}
\label{prob:r4poly}
Does there exist for every $k$ a convex simplicial polytope $P(k)$ in $\Rbb^4$ such that
in every coloring of vertices of $P(k)$ by $k$ colors we can find a triangular face whose vertices are monochromatic?
\end{problem}

Assuming the polytope $P(k)$ from Problem~\ref{prob:r4poly} exists, the set $X$ from Problem~\ref{prob:sep34} is obtained from $P(k)$ by making a one-point hole in an arbitrary point inside every
triangular face. Such a set $X$ cannot be covered by $k$ convex sets since otherwise one of the convex sets would contain three vertices of a triangular face.

The invisibility graph $I(X)$ can be colored by $13$ colors in the following way.
All the vertices of $P(k)$ get color $1$.
Tancer~\cite{Tancer13_duhove_simplexy} has shown that the edges of every $2$-dimensional simplicial complex PL-embeddable in $\mathbb{R}^3$ can be colored by $12$ colors so that for every triangular face the three edges on its boundary have three different colors. This applies, in particular, to the $2$-skeleton of every $4$-dimensional convex simplicial polytope. Thus we may use colors $2,3, \dots, 13$ to color the interiors of edges of $P(k)$ so that every edge gets only one color and edges in the same triangular face get distinct colors.
For each triangular face, the interior of the segment connecting the one-point hole with a point $p$ on the boundary is colored by the color of $p$.
All the remaining points of $X$ are isolated in $I(X)$ and thus may be colored arbitrarily.

The boundary complex of a $4$-dimensional convex simplicial polytope is a special case of a triangulation of $S^3$. If we relax the condition on polytopality in Problem~\ref{prob:r4poly} and ask only for a triangulation of $S^3$, then the answer is yes. Heise et al.~\cite{HPPT12_coloring_hypergraphs} constructed, for every $k$, a $2$-dimensional simplicial complex linearly embedded in $\mathbb{R}^3$ such that in every coloring of its vertices with $k$ colors at least one of the triangles is monochromatic.
We found the same simplicial complex independently, modifying Boris Bukh's construction, which was communicated to us by Martin Tancer. The vertices of the complex are placed on the moment curve and a suitable noncrossing subset of triangles is chosen for the faces. It remains to extend the embedded complex to a triangulation of the whole $\mathbb{R}^3$, or $S^3$; see~\cite[Lemma 6]{B59_3manifolds} or~\cite[Lemma or Theorem 5]{WN35_complexes}. 
We thank Karim Adiprasito for bringing these two references to our attention.

Gonska and Padrol~\cite{GoPa13_neighborly} constructed $n^{n+o(n)}$ combinatorially distinct simplicial neighborly $4$-polytopes obtained by lifting $3$-dimensional Delaunay triangulations. Each of the triangulations is obtained from a convex $n$-gon in the $xy$-plane by lifting its vertices in the direction of the $z$-axis, in a specified order, so that each vertex $v_i$ ``sees'' the upper envelope of the convex hull of $v_1, v_2, \dots, v_{i-1}$. It is easy to see that the resulting triangulation and the corresponding $4$-polytope can be colored with $3$ colors, since each triangular face contains an edge of the outerplanar graph that can be represented as the union of the boundaries of the projections of $\conv\{v_1, v_2, \dots, v_{i}\}$, $i=3,4,\dots,n$, to the $xy$-plane. We have no example of a simplicial $4$-polytope where $3$ colors are not sufficient.

\bibliographystyle{plain}

\end{document}